\documentclass[11pt,reqno]{amsart}
\usepackage{amsmath,amssymb,extarrows}
\usepackage{url}
\usepackage{tikz,enumerate}
\usepackage{diagbox}
\usepackage{appendix}
\usepackage{epic}

\usepackage{float} 

\usepackage{cite}

\usepackage{hyperref}
\usepackage{array}

\usepackage{booktabs}

\allowdisplaybreaks[4]

\newtheorem{theorem}{Theorem}
\newtheorem{corollary}[theorem]{Corollary}
\newtheorem{conj}[theorem]{Conjecture}
\newtheorem{lemma}[theorem]{Lemma}
\newtheorem{prop}[theorem]{Proposition}
\theoremstyle{definition}
\newtheorem{defn}{Definition}
\theoremstyle{remark}
\newtheorem{rem}{Remark}
\numberwithin{equation}{section}
\numberwithin{theorem}{section}
\numberwithin{defn}{section}

\begin{document}
\title[Sign Changes of Coefficients of Powers of the Infinite Borwein Product]
 {Sign Changes of Coefficients of Powers of the Infinite Borwein Product}

\author{Liuquan Wang}
\address{School of Mathematics and Statistics, Wuhan University, Wuhan 430072, Hubei, People's Republic of China}
\email{wanglq@whu.edu.cn;mathlqwang@163.com}




\begin{abstract}
We denote by $c_t^{(m)}(n)$ the coefficient of $q^n$ in the series expansion of $(q;q)_\infty^m(q^t;q^t)_\infty^{-m}$, which is the $m$-th power of the infinite Borwein product. Let $t$ and $m$ be positive integers with $m(t-1)\leq 24$. We  provide asymptotic formula for $c_t^{(m)}(n)$, and give characterizations of $n$ for which $c_t^{(m)}(n)$ is positive, negative or zero. We show that $c_t^{(m)}(n)$ is ultimately periodic in sign and conjecture that this is still true for other positive integer values of $t$ and $m$. Furthermore, we confirm this conjecture in the cases $(t,m)=(2,m),(p,1),(p,3)$ for arbitrary positive integer $m$ and prime $p$.
\end{abstract}

\maketitle
\tableofcontents

\section{Introduction and main results}\label{sec-intro}
For a positive integer $t$, we define the following infinite product and write its series expansion as
\begin{align}
G_t(q):=\prod\limits_{n=1}^\infty \frac{1-q^n}{1-q^{tn}}=\frac{(q;q)_\infty}{(q^t;q^t)_\infty}=\sum_{n=0}^\infty c_t(n)q^n.
\end{align}
Here and throughout this paper, we use the standard $q$-series notation:
\begin{align}
(a;q)_\infty&:=\prod\limits_{n=0}^\infty (1-aq^n), \quad |q|<1.
\end{align}

Andrews \cite{Andrews} mentioned that P. Borwein (1993) considered the series expansion of $G_p(q)$ for $p$ being a prime. The function $G_t(q)$ was thus called the \textit{infinite Borwein product} by Schlosser and Zhou \cite{Schlosser-Zhou}.

Andrews \cite{Andrews} studied the sign pattern of $c_p(n)$ and proved \cite[Theorem 2.1]{Andrews} that for all primes $p$, $c_p(n)$ and $c_p(n+p)$ have the sam sign for each $n\geq 0$, which he also wrote as
\begin{align}\label{eq-Andrews-sign}
c_p(n)c_p(n+p)\geq 0 \quad \text{for all $n\geq 0$}.
\end{align}
He mentioned that Garvan and Borwein have a different proof  of this result (unpublished). The statement in \cite[Theorem 2.1]{Andrews} is not in the best form as it does not tell us whether the signs of $c_p(n-p)$ and $c_p(n+p)$ are the same or not when $c_p(n)=0$. In fact, from the  proof of this result in \cite{Andrews}, Andrews actually proved a stronger result:  for all $n,k\geq 0$,
\begin{align}\label{eq-strong-Andrews-sign}
c_p(n)c_p(n+pk)\geq 0.
\end{align}

It is natural to consider powers of the infinite Borwein product. Let $m$ be a real number. We write the series expansion of $G_t^m(q)$ as
\begin{align}\label{ctmn-defn}
G_t^m(q)=\frac{(q;q)_\infty^m}{(q^t;q^t)_\infty^m}=\sum_{n=0}^\infty c_t^{(m)}(n)q^n.
\end{align}
In particular, for $t\in \{2,3,4,5,7,9,13\}$ and $m=\frac{24}{t-1}$, it is known that
\begin{align}\label{eq-hauptmodul}
q^{-1}(G_t(q))^{\frac{24}{t-1}}=\Big(\frac{\eta(\tau)}{\eta(t\tau)}\Big)^{\frac{24}{t-1}}
\end{align}
is a hauptmodul of the congruence subgroup $\Gamma_0(t)$. Here
$$\eta(\tau):=q^{1/24}(q;q)_\infty, \quad q=e^{2\pi i\tau}, \quad \mathrm{Im} \tau>0.$$
is the Dedekind eta function.

Hauptmoduls are important objects in the theory of modular forms. One aspect of their importance is that they serve as generators of function fields consisting of certain modular functions. For instance, every modular function invariant under $\Gamma_0(t)$ ($t\in \{2,3,4,5,7,9,13\}$) can be expressed as a rational function of the hauptmodul in \eqref{eq-hauptmodul}.

Surprisingly, the coefficients $c_t^{(m)}(n)$ also posses some nice sign pattern. For $t=2$, we have
\begin{align}
&\sum_{n=-1}^\infty c_2^{(24)}(n+1)q^n=\frac{\eta^{24}(\tau)}{\eta^{24}(2\tau)}=\frac{1}{q}\prod\limits_{n=1}^\infty \frac{(1-q^n)^{24}}{(1-q^{2n})^{24}}  \nonumber \\
&=\frac{1}{q}-24+276q-2048q^2+11202q^3-49152q^4+184024q^5-614400q^6+O(q^7).
\end{align}
Ohta \cite{Ohta} found that
\begin{align}\label{c2(n)-exp}
c_2^{(24)}(n+1)=\frac{1}{n}\Big(\sum_{r\in \mathbb{Z}} t(n-r^2)+\sum_{r\geq 1, r~\text{odd}} (-1)^nt(4n-r^2)+24\sum_{d|n, d~\text{odd}} d\Big).
\end{align}
Here $t(d)$ is the trace of singular moduli of discriminant $-d$. See \cite{Ohta} for explicit definition of $t(d)$. Matsusaka and Osanai \cite{MO} found similar formulas for $c_t^{(m)}(n)$ with $(t,m)\in \{(2,24),(3,12),(5,6)\}$. Based on these formulas and Laplace's method, they \cite{MO} proved that as $n\rightarrow \infty$,
\begin{align}
c_2^{(24)}(n+1)&\sim \frac{e^{2\pi \sqrt{n}}}{2n^{3/4}} (-1)^{n+1}, \label{c2-asymptotic} \\
c_3^{(12)}(n+1)&\sim \frac{e^{4\pi \sqrt{n}/3}}{\sqrt{6}n^{3/4}} \times \left\{\begin{array}{ll}
-1 & n\equiv 0,2\pmod{3}, \\
2 & n\equiv 1\pmod{3};
\end{array}\right. \label{c3-asymptotic} \\
c_5^{(6)}(n+1) &\sim \frac{e^{4\pi \sqrt{n}/5}}{\sqrt{10}n^{3/4}}\times \left\{\begin{array}{ll}
-1 & n\equiv 0 \pmod{5},\\
\frac{3+\sqrt{5}}{2} & n\equiv 1 \pmod{5}, \\
-1+\sqrt{5} & n\equiv 2 \pmod{5}, \\
-1-\sqrt{5} & n\equiv 3 \pmod{5}, \\
\frac{3-\sqrt{5}}{2} & n\equiv 4 \pmod{5}.
\end{array}\right. \label{c5-asymptotic}
\end{align}
This means that $c_2^{(24)}(n)$, $c_3^{(12)}(n)$ and $c_5^{(6)}(n)$ possesses a sign-change property for large $n$. Later Hu and Ye \cite{Hu-Ye} proved that for any $n\geq 0$,
\begin{align}
(-1)^{n}c_2^{(24)}(n)>0.
\end{align}
Their proof is based on careful analysis of the formula \eqref{c2(n)-exp}.

For prime $t\geq 3$ with $(t-1)|24$, i.e., $t\in \{3,5,7,11,13\}$, Hu and Ye \cite{Hu-Ye}  mentioned that in \cite{Hu-Ye-preparation} they will study the sign change  behaviors of the coefficients of $(G_t(q))^{24/(t-1)}$ by considering the generalized traces of singular moduli.

In 2019, Schlosser \cite{Schlosser} considered real powers of the infinite Borwein product and presented an interesting conjecture. For $\frac{9-\sqrt{73}}{2}\leq \delta \leq 1$ or $2\leq \delta \leq 3$, he conjectured that for all $n\geq 0$,
\begin{align}
c_3^{(\delta)}(3n)\geq 0, \quad c_3^{(\delta)}(3n+1)\leq 0, \quad c_3^{(\delta)}(3n+2)\leq 0.
\end{align}
In a recent work by Schlosser and Zhou \cite{Schlosser-Zhou}, a growth estimate for $c_3^{(\delta)}(n)$ for $\delta\in (0.227,3]$ was given by using the circle method. As a consequence, they were able to prove that for all integers $n\geq 158$ and $\delta \in [0.227,2.9999]$,
\begin{align}
c_3^{(\delta)}(n)c_3^{(\delta)}(n+3)>0.
\end{align}
For $(t,m)=(3,3)$, They \cite[Theorem 6]{Schlosser-Zhou} also proved the following cubic analogue of Andrews' result: for all primes $p$, $c_p^{(3)}(n)$ and $c_p^{(3)}(n+p)$ have the same sign for each $n\geq 0$, i.e.,
\begin{align}\label{eq-Schlosser-sign}
c_p^{(3)}(n)c_p^{(3)}(n+p)\geq 0.
\end{align}
Again, similar to Andrews' \eqref{eq-Andrews-sign}, this statement is not of the best form. In fact, from the proof given in \cite{Schlosser-Zhou},  \eqref{eq-Schlosser-sign} can be replaced by the following stronger assertion:
\begin{align}\label{eq-strong-Schlosser-sign}
c_p^{(3)}(n)c_p^{(3)}(n+pk)\geq 0, \quad \text{for all $n,k\geq 0$}.
\end{align}

For convenience, we introduce some concepts for sequences with nice sign patterns.
\begin{defn}\label{defn-UPS}
Given a sequence $\{a(n)\}_{n\geq 0}$ of real numbers. If there exists a positive number $t$ such that $a(n)a(n+tk)\geq 0$ holds for all sufficiently large $n$ and any positive integer $k$, then we call $\{a(n)\}_{n\geq 0}$ a \textit{sequence ultimately weakly periodic in sign} (weak UPS sequence for short). We call $t$ as a \textit{weak period of sign} for $\{a(n)\}_{n\geq 0}$. We also call the smallest one among the weak periods as the \textit{least weak period of sign} for $\{a(n)\}_{n\geq 0}$.
\end{defn}
The reason for using the adjective ``weak'' is that for a weak UPS sequence, we admit that 0 has the same sign with positive numbers and negative numbers. But in many cases, we may want to differ the sign of 0 from nonzero numbers. That is, we define a more strict sign as
\begin{align}
\mathrm{sgn}(x):=\left\{\begin{array}{ll}
1 & x>0,\\
0 & x=0, \\
-1 & x<0.
\end{array}\right.
\end{align}
We are more interested in sequences $\{a(n)\}$ for which $\{\mathrm{sgn}(a(n))\}$ is ultimately periodic.
\begin{defn}\label{defn-strongUPS}
Suppose $\{a(n)\}$ is a sequence of real numbers. If there exists some positive integer $t$ such that $\mathrm{sgn}(a(n))=\mathrm{sgn}(a(n+t))$ for all sufficiently large $n$, then we call $\{a(n)\}_{n\geq 0}$ a \textit{sequence ultimately periodic in sign} (UPS sequence for short) and call $t$ as its \textit{period of sign}. The smallest period of sign is called its \textit{least period of sign}.
\end{defn}
Before we go further, we make some remarks on these concepts.
\begin{enumerate}
\item It is clear that a UPS sequence is also a weak UPS sequence. But the converse is not true. For example, if we define $a(n)=1$ for $n=k^2$ ($k\in \mathbb{Z}$) and 0 otherwise. Then $\{a(n)\}$ is a weak UPS sequence with least weak period of sign 1. But it is not a UPS sequence.

\item It is easy to show that for a UPS sequence with least period of sign $t$, all the periods of sign for this sequence are $kt$ with $k$ being any positive integer. This is not true for weak UPS sequence. For example, if we define for $n\geq 0$ that $a(6n)=1$, $a(6n+1)=-1$ and $a(6n+r)=0$ ($r=2,3,4,5$), then $\{a(n)\}$ is a weak UPS sequence with least weak period of sign 2. But 3 is also a weak period of sign for it and $2$ does divide $3$.

\item For a UPS sequence, the period of sign is also a weak period of sign but the converse is not true. For instance, consider the sequence $a(n)=\frac{1}{2}(1+(-1)^n)$ ($n=0,1,2,\cdots $). Then its least period of sign is 2 but its least weak period of sign is $1$.
\end{enumerate}

For any prime $p$, thanks to the works of Andrews \cite{Andrews} and Schlosser and Zhou \cite{Schlosser-Zhou}, we now know that $\{c_p^{(1)}(n)\}$ and $\{c_p^{(3)}(n)\}$ are both weak UPS sequences with period $p$. In fact, later we will show that they are actually UPS sequences with $p$ as the least period of sign. By the work of Matsusaka and Osanai \cite{MO}, from \eqref{c2-asymptotic}--\eqref{c5-asymptotic} we know that $\{c_t^{(m)}(n)\}$ is a UPS sequence for $(t,m)\in \{(2,23),(3,12),(5,6)\}$.

We should be aware that not all coefficients generated by infinite products are (weak) UPS sequences. For instance, Euler's pentagonal number theorem and Jacobi's identity state that
\begin{align}
(q;q)_\infty=\sum_{n=-\infty}^\infty (-1)^nq^{n(3n+1)/2}, \quad \text{and}  \label{eq-pentagonal}\\
(q;q)_\infty^3=\sum_{n=-\infty}^\infty (-1)^n(2n+1)q^{n(n+1)/2}, \label{eq-Jacobi}
\end{align}
respectively.  It is then clear that the coefficients of $(q;q)_\infty$ or $(q;q)_\infty^3$ do not form a weak UPS sequence.

In this paper, we will discuss whether $\{c_t^{(m)}(n)\}$ form a UPS sequence or not for positive integers $t$ and $m$. To achieve this goal, we shall investigate its sign change behaviors. We will show that for all positive integers $t$ and $m$ with $m(t-1)\leq 24$, $\{c_t^{(m)}(n)\}$ is a UPS sequence. Furthermore, we give explicit characterizations of its sign-patterns by determining the sign of $c_t^{(m)}(n)$ for each $n$. To state the results, we need more notations.

For positive integer $k$ and integer $h$ with $(h,k)=1$, the Dedekind sum $s(h,k)$ is defined by
\begin{align}\label{Dedekind-sum}
s(h,k):=\sum_{r=1}^{k-1} \frac{r}{k}\left(\frac{hr}{k}-\left\lfloor \frac{hr}{k} \right\rfloor -\frac{1}{2} \right).
\end{align}
Here $\lfloor x\rfloor $ denotes the integer part of $x$ and we agree that $s(h,1)=0$. We define
\begin{align}\label{alpha-defn}
\alpha_t^{(m)}(n):=\sum_{\begin{smallmatrix}
0\leq h <t \\ \mathrm{gcd}(h,t)=1
\end{smallmatrix}}\exp\left(-m\pi i s(h,t)-\frac{2\pi inh}{t}\right).
\end{align}
It is easy to see that $\alpha_t^{(m)}(n)$ is a periodic function with period $t$, i.e., for all $n\geq 0$,
\begin{align}\label{alpha-periodic}
\alpha_t^{(m)}(n)=\alpha_t^{(m)}(n+t).
\end{align}

The modified Bessel function of the first kind is defined as
\begin{align}
I_s(x):=\sum_{m=0}^\infty \frac{1}{m!\Gamma(m+s+1)}\left(\frac{x}{2}\right)^{2m+s},
\end{align}
where $\Gamma(s)=\int_0^\infty e^{-x}x^{s-1}dx$ ($\mathrm{Re}~ s>0$) is the gamma function. In particular, when $s=n$ is an integer, we have
\begin{align}\label{Bessel-int}
I_n(x)=\frac{1}{\pi}\int_0^{\pi} e^{x\cos \theta}\cos(n\theta)d\theta.
\end{align}

We establish the following asymptotic formula for $c_t^{(m)}(n)$, which will be a key in recognizing the signs of the coefficients.
\begin{theorem}\label{thm-full-asymptotic}
Let $t$ and $m$ be positive integers with $m(t-1)\leq 24$. Let
\begin{align}
A&:=\max \left\{\frac{t}{\gcd(t,\ell)}: 1\leq \ell \leq t \,\, \text{and} \,\, \gcd(t,\ell)>\sqrt{t}\right\}, \label{A-defn} \\
M&:=\max \left\{\frac{\sqrt{m(t-1)}}{2t}\right\}\bigcup \left\{\frac{1}{\ell}\sqrt{\frac{m\left({\gcd}^2(t,\ell)-t\right)}{t}}: 1\leq \ell< t \,\, \text{and} \,\, \gcd(t,\ell)>\sqrt{t} \right\}. \label{M-defn}
\end{align}
Then
\begin{align}\label{eq-full}
c_t^{(m)}(n)=\frac{2\pi \mu^{\frac{1}{2}}}{t}(n-\mu)^{-\frac{1}{2}} \alpha_t^{(m)}(n)I_{-1}\left(\frac{4\pi\sqrt{\mu(n-\mu)}}{t} \right)+E_t^{(m)}(n),
\end{align}
where $\mu=m(t-1)/24$, $|E_t^{(m)}(n)|\leq \overline{E}_t^{(m)}(n)$ with
\begin{align}\label{error-exp}
\overline{E}_t^{(m)}(n):= & \frac{\pi^{\frac{7}{4}}}{2^{\frac{3}{4}}}A^{\frac{m}{2}}\mu^{\frac{1}{4}}\exp\left(\frac{\sqrt{6}\pi M}{3}\sqrt{n-\mu} \right)
+2te^{2+8\pi\mu}  \nonumber \\
&+2e^2t^{\frac{m}{2}+1}\exp\left(\pi \mu +m\Big(\frac{e^{-\pi}}{(1-e^{-\pi})^2}+\frac{e^{-\pi/t}}{(1-e^{-\pi/t})^2} \Big) \right).
\end{align}
\end{theorem}
The proof of this theorem will rely on the work of Chern \cite{Chern}, which is based on the circle method. We remark here that for $t$ being a prime and $\delta \in (0,24/(t-1)]$, an asymptotic formula for $c_t^{(\delta)}(n)$ with different form was given by Schlosser and Zhou \cite{Schlosser-Zhou} through the circle method.

Recall that for fixed $s$, when $|\arg x|<\frac{\pi}{2}$, we have (see \cite[p. 377, (9.7.1)]{AS}, for example.)
\begin{align}
I_s(x)\sim \frac{e^x}{\sqrt{2\pi x}}\left(1-\frac{4s^2-1}{8x}+\frac{(4s^2-1)(4s^2-9)}{2!(8x)^2}-\cdots \right). \label{Bessel-general-asymptotic}
\end{align}
This together with Theorem \ref{thm-full-asymptotic} leads to the following simplified asymptotic formula, which is easier for us to understand the growth rate of $c_t^{(m)}(n)$.
\begin{corollary}\label{cor-asymptotic}
For any positive integers $t$ and $m$ with $m(t-1)\leq 24$, let $\mu=m(t-1)/24$.  As $n\rightarrow \infty$, we have
\begin{align}\label{eq-asymptotic}
c_t^{(m)}(n)=\frac{\mu^{\frac{1}{4}}}{\sqrt{2t}}\frac{e^{4\pi\sqrt{\mu(n-\mu)}/t}}{(n-\mu)^{3/4}}\left(\alpha_t^{(m)}(n)+O(n-\mu)^{-\frac{1}{2}}\right).
\end{align}
\end{corollary}
When $(t,m)\in \{(2,24),(3,12),(5,6)\}$, by a direct calculation of $\alpha_t^{(m)}(n)$ and Corollary \ref{cor-asymptotic}, we recover \eqref{c2-asymptotic}--\eqref{c5-asymptotic}.

Note that when $\alpha_t^{(m)}(n)\neq 0$,  Theorem \ref{thm-full-asymptotic} and Corollary \ref{cor-asymptotic} give exact orders of $c_t^{(m)}(n)$. We can see that the sign of $c_t^{(m)}(n)$ will be the same as $\alpha_t^{(m)}(n)$ when $n$ is large enough. When $\alpha_t^{(m)}(n)=0$, \eqref{eq-full} and \eqref{eq-asymptotic} give upper bounds for $c_t^{(m)}(n)$. From \eqref{alpha-periodic} we know $\alpha_t^{(m)}(n)$  depends only on the residue of $n$ modulo $t$. Thus we define
\begin{align}
P_t^{(m)}:=\{r: 0\leq r<t \,\, \text{and} \,\, \alpha_t^{(m)}(r)>0\}, \\
Z_t^{(m)}:=\{r: 0\leq r<t \,\, \text{and} \,\, \alpha_t^{(m)}(r)=0\}, \\
N_t^{(m)}:=\{r: 0\leq r<t \,\, \text{and} \,\, \alpha_t^{(m)}(r)<0\}.
\end{align}
It is clear that
\begin{align}
P_t^{(m)}\cup Z_t^{(m)}\cup N_t^{(m)}=\{0,1,\cdots,t-1\}.
\end{align}
These sets can be found by computing the values of $\alpha_t^{(m)}(r)$ using \eqref{alpha-defn}. We have done so for all positive integer pairs $(t,m)$ satisfying $m(t-1)\leq 24$ and list these sets in the Appendix.

From Theorem \ref{thm-full-asymptotic} or Corollary \ref{cor-asymptotic} we get the following consequence in the case $\alpha_t^{(m)}(n)\neq 0$.
\begin{theorem}\label{thm-sign}
For any positive integers $t,m$ with $m(t-1)\leq 24$, there exists some positive integer $n_0$ such that when $n\geq n_0$,
\begin{align}
c_t^{(m)}(tn+r)>0, \quad r \in P_t^{(m)}, \label{eq-positive} \\
c_t^{(m)}(tn+r)<0, \quad r \in N_t^{(m)}.  \label{eq-negative}
\end{align}
\end{theorem}

Furthermore, when $\alpha_t^{(m)}(n)\neq 0$, there are only finitely many $n$ for which $c_t^{(m)}(n)$ does not have the same sign as $\alpha_t^{(m)}(n)$. We denote the set of such exceptional $n$'s as $\mathcal{E}_t^{(m)}$. Given any pair of positive integer $(t,m)$ with $m(t-1)\leq 24$, it is possible to find $\mathcal{E}_t^{(m)}$ explicitly based on Theorem \ref{thm-full-asymptotic}.  This requires some tedious calculus and computation. At the end of Section \ref{sec-asymptotic}, we will treat the case when $t$ is a prime and use it to illustrate the procedure. We have found $\mathcal{E}_t^{(m)}$ for primes $t\leq 24$ and record them in the Appendix.  The case for $t$ being a composite number can be treated similarly. But the computations involved are more complicated and we need to discuss case by case. Thus we do not pursue it here.

Theorem \ref{thm-sign} already means that $\{c_t^{(m)}(n)\}$ is a weak UPS sequence. To see whether it is a UPS sequence or not, we need to identify those $n$ such that $c_t^{(m)}(n)=0$. The asymptotic formula in \eqref{eq-asymptotic} indicates that the residue of such $n$ modulo $t$ must be in $Z_t^{(m)}$. This turns out to be true except for three special cases.
\begin{theorem}\label{thm-zero}
For any positive integers $t,m$ with $m(t-1)\leq 24$,  we have for all $n\geq 0$ that
\begin{align}
c_t^{(m)}(tn+r)=0, \quad r \in Z_t^{(m)} \label{eq-zero}
\end{align}
except for $c_4^{(8)}(1)=-8$, $c_9^{(3)}(1)=-3$ and three special cases with
$$(t,m) \in \{(3,9),(4,4),(5,5)\}.$$
In these cases, we have for any $n\geq 0$,
\begin{align}
&c_3^{(9)}(9n)>0, \quad c_3^{(9)}(9n+3)<0, \quad c_3^{(9)}(9n+6)=0; \label{eq-exc-1}\\
&c_4^{(4)}(8n+r)>0 \quad (r\in \{0,2\}),  \quad c_4^{(4)}(8n+r)<0, \quad (r\in \{4,6\}); \label{eq-exc-2} \\
&c_5^{(5)}(25n)>0, \quad  c_5^{(5)}(25n+r)<0 \quad (r\in \{5,10\}),  \quad c_5^{(5)}(25n+r)=0  \quad (r\in \{15,20\}). \label{eq-exc-3}
\end{align}
\end{theorem}
As an example, from the case $m=5$ in Table \ref{tab-5} in the Appendix we read that
\begin{align}
c_5^{(4)}(5n+r)&>0, \quad r\in \{0,2,3\},\\
c_5^{(4)}(5n+r)&<0, \quad r \in \{1,4\}.
\end{align}

From Theorems \ref{thm-sign} and \ref{thm-zero}, we finally arrive at the following conclusion.
\begin{theorem}\label{thm-UPS}
For any positive integers $t,m$ with $m(t-1)\leq 24$, $\{c_t^{(m)}(n)\}_{n\geq 0}$ is a UPS sequence. Its least period of sign is $9$, $8$, $25$ for $(t,m)=(3,9),(4,4)$ and $(5,5)$, respectively and $t$ otherwise.
\end{theorem}

It appears that similar sign patterns exist for $c_t^{(m)}(n)$ when $m(t-1)>24$. Based on Theorem \ref{thm-UPS} and further numerical data, we make the following
\begin{conj}\label{conj}
For any positive integers $t$ and $m$, $\{c_t^{(m)}(n)\}_{n\geq 0}$ is a UPS sequence with least period of sign divisible by $t$.
\end{conj}
Due to the restriction of the condition in Chern's asymptotic formula (see \eqref{ineq-cond}), we are not able to prove this conjecture. A possible way to attack it is to generalize Chern's result and drop the restrictions, which might be possible as noted by Schlosser and Zhou
\cite[Remark 2]{Schlosser-Zhou}.

Nevertheless, in addition to Theorem \ref{thm-UPS}, we can further prove Conjecture \ref{conj} in some more cases.
\begin{theorem}\label{thm-conj-partial}
Conjecture \ref{conj} holds in any of the following cases:
\begin{enumerate}[$(1)$]
\item $t=2$ and $m$ is a positive integer;
\item $t=p$ is a prime and $m=1$ or $3$.
\end{enumerate}
\end{theorem}


The paper is organized as follows. In Section \ref{sec-asymptotic}, we will first prove Theorem \ref{thm-full-asymptotic}. Then we discuss the case when the coefficients $c_t^{(m)}(n)$ do not vanish and provide a proof for Theorem \ref{thm-sign}. In Section \ref{sec-qseries} we will mainly use $q$-series techniques to prove Theorems \ref{thm-zero} and \ref{thm-UPS}. Section \ref{sec-conj} is devoted to discussions of Conjecture \ref{conj} and we will prove Theorem \ref{thm-conj-partial}.  Finally, in the Appendix, we provide 23 tables which record the explicit forms of the sets $P_t^{(m)}, N_t^{(m)}$, $Z_t^{(m)}$ and $\mathcal{E}_t^{(m)}$ (for $t$ prime).

\section{Asymptotic formula and nonvanishing coefficients}\label{sec-asymptotic}

In this section, we will first prove Theorem \ref{thm-full-asymptotic}, and then we give a proof for Theorem \ref{thm-sign}.

To deduce the asymptotic formula of $c_t^{(m)}(n)$, we need a result of Chern \cite{Chern}. Before stating it, we introduce some  notations in Chern's work \cite{Chern}.

Let $\boldsymbol{m}=(m_1,\cdots, m_R)$ be a sequence of $R$ distinct positive integers and $\boldsymbol{\delta}=(\delta_1,\cdots,\delta_R)$ be a sequence of $R$ nonzero integers. Let
\begin{align}
&\Delta_1=-\frac{1}{2}\sum_{r=1}^R \delta_r, \quad \Delta_2=\sum_{r=1}^Rm_r\delta_r, \\
&\Delta_3(k)=-\sum_{r=1}^R\frac{\delta_r \gcd^2(m_r,k)}{m_r}, \quad \Delta_4(k)=\prod\limits_{r=1}^R \left(\frac{m_r}{\gcd(m_r,k)} \right)^{-\frac{\delta_r}{2}}.
\end{align}
Let $L=\mathrm{lcm} (m_1,\cdots,m_R)$ and we separate the set $\{1,2,\cdots,L\}$ into two disjoint subsets:
\begin{align*}
\mathcal{L}_{>0}:=\{1\leq \ell \leq L: \Delta_3(\ell)>0\}, \\
\mathcal{L}_{\leq 0}:=\{1\leq \ell \leq L: \Delta_3(\ell)\leq 0\}.
\end{align*}
We also denote
\begin{align}
\omega_{h,k}=\exp\left(-\pi i \sum_{r=1}^R\delta_r \cdot s\left(\frac{m_rh}{\gcd(m_r,k)},\frac{k}{\gcd(m_r,k)} \right)  \right),
\end{align}
where $s(d,c)$ is the Dedekind sum defined in \eqref{Dedekind-sum}.

Consider the infinite product
\begin{align}\label{eta-prod}
\prod\limits_{r=1}^R(q^{m_r};q^{m_r})_\infty^{\delta_r}=\sum_{n=0}^\infty g(n)q^n.
\end{align}
Using the circle method, Chern obtained an asymptotic formula for $g(n)$ when $\Delta_1\leq 0$.
\begin{theorem}\label{thm-Chern}
(Cf.\ \cite[Theorem 1.1]{Chern}.) If $\Delta_1\leq 0$ and the inequality
\begin{align}\label{ineq-cond}
\min_{1\leq r \leq R} \left(\frac{\mathrm{gcd}^2(m_r,\ell)}{m_r} \right)\geq \frac{\Delta_3(\ell)}{24}
\end{align}
holds for all $1 \leq \ell \leq L$, then for positive integers $n>-\Delta_2/24$, we have
\begin{align}\label{g(n)-asymptotic}
g(n)=&E(n)+2\pi \sum_{\ell\in \mathcal{L}_{>0}}\Delta_4(\ell)\left(\frac{24n+\Delta_2}{\Delta_3(\ell)} \right)^{-\frac{\Delta_1+1}{2}} \nonumber \\
&\times \sum_{\begin{smallmatrix} 1\leq k \leq N \\ k\equiv \ell \!\!\! \pmod{L} \end{smallmatrix}} \frac{1}{k}\sum_{\begin{smallmatrix}
0\leq h <k \\ \mathrm{gcd}(h,k)=1
\end{smallmatrix}} \omega_{h,k}e^{-\frac{2\pi inh}{k}}I_{-\Delta_1-1}\left(\frac{\pi}{6k}\sqrt{\Delta_3(\ell)(24n+\Delta_2)} \right),
\end{align}
where
\begin{align}
N=\left\lfloor \sqrt{2\pi\left(n+\frac{\Delta_2}{24}\right)}\right\rfloor,
\end{align}
\begin{align}
E(n)\ll_{\boldsymbol{m},  \boldsymbol{\delta}} \left\{\begin{array}{ll}
1 & \text{if $\Delta_1=0$}, \\
\left(n+\frac{\Delta_2}{24}\right)^{1/4} &\text{if $\Delta_1=-\frac{1}{2}$}, \\
\left(n+\frac{\Delta_2}{24}\right)^{1/2}\log\left(n+\frac{\Delta_2}{24}\right) &\text{if $\Delta_1=-1$}, \\
\left(n+\frac{\Delta_2}{24}\right)^{-\Delta_1-1/2} &\text{if $\Delta_1\leq -\frac{3}{2}$}.
\end{array}\right.
\end{align}
\end{theorem}
Let $\zeta(s)$ be the Riemann zeta-function. For $\Delta\in \frac{1}{2}\mathbb{Z}_{\leq 0}$ and $x\geq 1$, as in \cite[Eq.\ (1.16)]{Chern}, we define
 \begin{align}\label{Xi-function}
\mathsf{\Xi}_\Delta(x):=\left\{\begin{array}{ll}
1 & \text{if $\Delta=0$}, \\
2x^{1/2} & \text{if $\Delta=-\frac{1}{2}$}, \\
x(\log x+1) &\text{if $\Delta=-1$}, \\
\zeta(-\Delta)x^{-2\Delta-1} & \text{otherwise}.
\end{array}\right.
 \end{align}
Chern gave an explicit bound for $E(n)$ (see \cite[Eq.\ (3.10)]{Chern}):
\begin{align}\label{E(n)-explicit}
&|E(n)|\leq \frac{2^{-\Delta_1}\pi^{-1}N^{-\Delta_1+2}}{n+\frac{\Delta_2}{24}}
e^{2\pi\left(n+\frac{\Delta_2}{24}\right)/N^2}\sum_{\ell \in \mathcal{L}_{>0}}e^{\frac{\pi \Delta_3(\ell)}{3}}   +2e^{2\pi\left(n+\frac{\Delta_2}{24}\right)/N^2} \mathsf{\Xi}_{\Delta_1}(N) \times  \notag \\
& \left(\sum_{1\leq \ell \leq L} \Delta_4(\ell)
\exp\left(\frac{\pi \Delta_3(\ell)}{24}+\sum_{r=1}^R \frac{|\delta_r|e^{-\pi \gcd^2(m_r,\ell)/m_r}}{(1-e^{-\pi \gcd^2(m_r,\ell)/m_r})^2} \right) -\sum_{\ell \in \mathcal{L}_{>0}} \Delta_4(\ell)e^{\frac{\pi \Delta_3(\ell)}{24}}\right).
\end{align}

We also need the following inequalities for $I_{-1}(x)$.
\begin{lemma}\label{lem-Bessel-ineq}
For $x>0$, we have
\begin{align}\label{I-1-upper}
I_{-1}(x)<\sqrt{\frac{\pi}{8}}\frac{e^x}{\sqrt{x}}.
\end{align}
For $x\geq 3$, we have
\begin{align}\label{I-1-lower}
I_{-1}(x)>\frac{1}{10}\frac{e^x}{\sqrt{x}}.
\end{align}
\end{lemma}

\begin{proof}
From \eqref{Bessel-int} we have
\begin{align}\label{I-1-start}
I_{-1}(x)&=\frac{1}{\pi}\int_{0}^\pi e^{x\cos \theta}\cos \theta d\theta =\frac{1}{\pi}\int_0^\pi e^{x\cos \theta}(1-2\sin^2\frac{\theta}{2})d\theta \notag \\
&=I_0(x)-\frac{2}{\pi}\int_0^\pi e^{x\cos\theta}\sin^2\frac{\theta}{2}d\theta.
\end{align}
It follows that for $x>0$,
\begin{align}
I_{-1}(x)<I_0(x)<\sqrt{\frac{\pi}{8}}\frac{e^x}{\sqrt{x}},
\end{align}
where the second inequality follows from \cite[Eq.\ (3.13)]{CTW}.

Next, we need the following well-known inequality: for $0\leq u \leq \frac{\pi}{2}$,
\begin{align}\label{sin-ineq}
\frac{2}{\pi} u\leq \sin u \leq u.
\end{align}
Note that
\begin{align}\label{I-1-middle}
&\int_0^\pi e^{x\cos \theta}\sin^2\frac{\theta}{2} d\theta \notag =\int_0^\pi e^{x(1-2\sin^2\frac{\theta}{2})}\sin^2\frac{\theta}{2}d\theta \quad \text{(set $\theta =2u$)} \notag \\
=&2e^x\int_0^{\frac{\pi}{2}} e^{-2x\sin^2 u}\sin^2 udu  \quad  \text{(use \eqref{sin-ineq})}   \notag \\
\leq &2e^x \int_{0}^{\frac{\pi}{2}} e^{-8xu^2/\pi^2}u^2d u \quad  (\text{set $u=\frac{\pi}{\sqrt{8x}}\sqrt{t}$}) \notag \\
\leq & e^x\left(\frac{\pi}{\sqrt{8x}}\right)^3\int_0^{\infty} e^{-t}t^{1/2} dt = \frac{\pi^{7/2}}{32\sqrt{2}}\frac{e^x}{x^{3/2}}.
\end{align}
By \cite[Eq.\ (3.14)]{CTW} we have for $x\geq 1$ that
$$I_0(x)>\frac{4\sqrt{2}}{5\pi}\frac{e^x}{\sqrt{x}}.$$
Therefore, when $x\geq 3$, by \eqref{I-1-start}, \eqref{I-1-middle} and the above inequality, we have
\begin{align*}
I_{-1}(x)>\frac{e^x}{\sqrt{x}}\left( \frac{4\sqrt{2}}{5\pi}-\frac{\pi^{5/2}}{16\sqrt{2}}\frac{1}{x}\right)\geq \frac{e^x}{\sqrt{x}}\left( \frac{4\sqrt{2}}{5\pi}-\frac{\pi^{5/2}}{48\sqrt{2}}\right) >\frac{1}{10}\frac{e^x}{\sqrt{x}}. \quad \quad \qedhere
\end{align*}
\end{proof}

\begin{proof}[Proof of Theorem \ref{thm-full-asymptotic}]
Let the infinite products in \eqref{eta-prod} be
$$G_t^{m}(q)=\frac{(q;q)_\infty^m}{(q^t;q^t)_\infty^m}.$$
 We have
\begin{align*}
&\boldsymbol{m}=(1,t), \quad \boldsymbol{\delta}=(m,-m), \quad \Delta_1=0, \quad \Delta_2=m(1-t)=-24\mu, \\
&\Delta_3(k)=m\left(\frac{\gcd^2(t,k)}{t}-1 \right), \quad \Delta_4(k)=\left(\frac{t}{\gcd(t,k)}\right)^{\frac{m}{2}}, \\
&L=t, \quad \mathcal{L}_{>0}=\left\{\ell: 1\leq \ell \leq t \,\, \text{and} \,\, \gcd(t,\ell)>\sqrt{t}
\right\}, \quad \mathcal{L}_{\leq 0}=\mathcal{L}\backslash \mathcal{L}_{>0}.
\end{align*}
The condition \eqref{ineq-cond} is satisfied since $m(t-1)\leq 24$.

From \eqref{g(n)-asymptotic} we have
\begin{align}
c_t^{(m)}(n)&=2\pi\Delta_4(t)\left(\frac{24n+\Delta_2}{\Delta_3(t)}  \right)^{-\frac{1}{2}} I_{-1}\left(\frac{\pi}{6t}\sqrt{\Delta_3(t)(24n+\Delta_2)}  \right)\cdot \frac{1}{t}\alpha_t^{(m)}(n) \nonumber \\
&\quad \quad \quad +E(n)+E^*(n) \nonumber \\
&=\frac{2\pi \mu^{\frac{1}{2}}}{t}(n-\mu)^{-\frac{1}{2}} \alpha_t^{(m)}(n)I_{-1}\left(\frac{4\pi\sqrt{\mu(n-\mu)}}{t} \right)+E(n)+E^*(n).
\end{align}
Here $E(n)$ was inherited from \eqref{g(n)-asymptotic} and
\begin{align}\label{error-E-first}
E^*(n)=&2\pi \sum_{\ell\in \mathcal{L}_{>0}}\Delta_4(\ell)\left(\frac{24n+\Delta_2}{\Delta_3(\ell)} \right)^{-\frac{1}{2}} \nonumber \\
&\times \sum_{\begin{smallmatrix} 1\leq k \leq N \\ k\equiv \ell \!\!\! \pmod{t} \\(\ell,k)\neq (t,t) \end{smallmatrix}} \frac{1}{k}\sum_{\begin{smallmatrix}
0\leq h <k \\ \mathrm{gcd}(h,k)=1 
\end{smallmatrix}} \omega_{h,k}e^{-\frac{2\pi inh}{k}}I_{-1}\left(\frac{\pi}{6k}\sqrt{\Delta_3(\ell)(24n+\Delta_2)} \right).
\end{align}

Now we give an upper bound for $E^*(n)$.
Note that for $\ell \in \mathcal{L}_{>0}$, we have
\begin{align}
\Delta_3(\ell)=\frac{m}{t}\left({\gcd}^2(t,\ell)-t\right)\leq m(t-1)=24\mu,  \quad \Delta_4(\ell)\leq A^{\frac{m}{2}}.
\end{align}
We have
\begin{align}\label{error-start}
&\Big|E^*(n)   \Big| \nonumber \\
&\leq 2\pi A^{\frac{m}{2}} (24n+\Delta_2)^{-\frac{1}{2}}\sum_{\begin{smallmatrix}\ell\in \mathcal{L}_{>0} \end{smallmatrix}}(\Delta_3(\ell))^{\frac{1}{2}}  \sum_{\begin{smallmatrix} l\leq k \leq N \\ k\equiv \ell \!\!\! \pmod{t} \\(\ell,k)\neq (t,t) \end{smallmatrix}} \left|I_{-1}\left(\frac{\pi}{6k}\sqrt{\Delta_3(\ell)(24n+\Delta_2)}\right) \right| \nonumber \\
&\leq 2\pi A^{\frac{m}{2}} (24n+\Delta_2)^{-\frac{1}{2}}\sum_{\begin{smallmatrix} \ell\in \mathcal{L}_{>0}  \end{smallmatrix}}(\Delta_3(\ell))^{\frac{1}{2}}  \sqrt{\frac{\pi}{8}}\sum_{\begin{smallmatrix} l\leq k \leq N \\ k\equiv \ell \!\!\! \pmod{t} \\(\ell,k)\neq (t,t) \end{smallmatrix}} \frac{\exp\left(\frac{\pi}{6k}\sqrt{\Delta_3(\ell)(24n+\Delta_2)} \right)}{\sqrt{\frac{\pi}{6k}}\Delta_3(\ell)^{\frac{1}{4}}(24n+\Delta_2)^{\frac{1}{4}}} \nonumber \\
&\leq \sqrt{3}\pi A^{\frac{m}{2}} (24n+\Delta_2)^{-\frac{3}{4}}\sum_{\begin{smallmatrix}\ell\in \mathcal{L}_{>0} \end{smallmatrix}}(\Delta_3(\ell))^{\frac{1}{4}}   \sum_{\begin{smallmatrix} l\leq k \leq N \\ k\equiv \ell \!\!\! \pmod{t} \\ (\ell,k)\neq (t,t) \end{smallmatrix}} \exp\left(\frac{\pi}{6k}\sqrt{\Delta_3(\ell)(24n+\Delta_2)}\right) \sqrt{k} \nonumber \\
&\leq \sqrt{3}\pi A^{\frac{m}{2}} (24n-24\mu)^{-\frac{3}{4}}(24\mu)^{\frac{1}{4}}\sum_{\begin{smallmatrix} \ell\in \mathcal{L}_{>0} \end{smallmatrix}}  \sum_{\begin{smallmatrix} l\leq k \leq N \\ k\equiv \ell \!\!\! \pmod{t} \\(\ell,k)\neq (t,t) \end{smallmatrix}} \exp\left(\frac{\pi}{6k}\sqrt{\Delta_3(\ell)(24n-24\mu)}\right) \sqrt{k} \nonumber \\
&\leq \frac{\sqrt{2}\pi}{4}A^{\frac{m}{2}}(n-\mu)^{-\frac{3}{4}}\mu^{\frac{1}{4}}\exp\left(\frac{\sqrt{6}\pi M}{3}\sqrt{n-\mu} \right) \sum_{1\leq k \leq N} \sqrt{k} \nonumber \\
&\leq \frac{\sqrt{2}\pi}{4}A^{\frac{m}{2}}(n-\mu)^{-\frac{3}{4}}\mu^{\frac{1}{4}}\exp\left(\frac{\sqrt{6}\pi M}{3}\sqrt{n-\mu} \right) N^{\frac{3}{2}} \nonumber \\
&\leq \frac{\pi^{\frac{7}{4}}}{2^{\frac{3}{4}}}A^{\frac{m}{2}}\mu^{\frac{1}{4}}\exp\left(\frac{\sqrt{6}\pi M}{3}\sqrt{n-\mu} \right).
\end{align}
Here for the second inequality we used Lemma \ref{lem-Bessel-ineq}, and for the last third inequality we used the definition of $M$ in \eqref{M-defn}.

Next, since
$$\Delta_4(\ell)\leq t^{\frac{m}{2}} \quad (1\leq \ell \leq t), \quad \text{and} \quad 2\pi (n-\mu)/N^2\leq 2 \quad (n\geq 3),$$
from \eqref{E(n)-explicit} we have
\begin{align}
|E(n)|\leq 2te^{2+8\pi \mu}+2e^2t^{1+\frac{m}{2}}\exp\left(\pi \mu+m\left(\frac{e^{-\pi}}{(1-e^{-\pi})^2}+\frac{e^{-\pi/t}}{(1-e^{-\pi/t})^2}\right)\right).
\end{align}
This completes the proof.
\end{proof}
\begin{rem}\label{rem-error}
Note that for $\ell\in \mathcal{L}_{>0}$ and $\ell<t$, we have
\begin{align*}
\frac{\Delta_3(\ell)}{\ell^2}=\frac{m}{t}\left(\frac{\gcd^2(t,\ell)}{\ell^2}-\frac{t}{\ell^2}  \right)<\frac{m}{t}\left(1-\frac{1}{t}\right)=\frac{\Delta_3(t)}{t^2}.
\end{align*}
Hence
\begin{align*}
M< \frac{\sqrt{\Delta_3(t)}}{t}=\frac{\sqrt{24\mu}}{t}.
\end{align*}
Therefore, when $\alpha_t^{(m)}(n)\neq 0$, $E^*(n)$ contributes to the error term.
\end{rem}

\begin{proof}[Proof of Corollary \ref{cor-asymptotic}]
This follows from \eqref{Bessel-general-asymptotic} and Theorem \ref{thm-full-asymptotic}.
\end{proof}

\begin{proof}[Proof of Theorem \ref{thm-sign}]
The assertions follow readily from \eqref{eq-asymptotic} since when  $n$ is large enough and $\alpha_t^{(m)}(n)\neq 0$, the sign of $c_t^{(m)}(n)$ will be the same as the sign of $\alpha_t^{(m)}(n)$.
\end{proof}

As promised in the introduction, now we illustrate an approach for finding the exceptional set $\mathcal{E}_t^{(m)}$. We will only consider the case when $t$ is a prime.

Let $t$ be a prime less than $24$ and $1\leq m \leq 24/(t-1)$. In this case, we have
\begin{align}
A=1, \quad M=\frac{\sqrt{m(t-1)}}{2t}=\frac{\sqrt{6\mu}}{t}.
\end{align}
With the help of Mathematica,  from \eqref{alpha-defn} and by direct computation it is easy to verify that $|\alpha_t^{(m)}(n)| \geq 0.1$ for any positive integer $n\equiv r$ (mod $t$) with $r\in P_{t}^{(m)}\cup N_t^{(m)}$.

We have already seen that: when $\alpha_t^{(m)}(n)\neq 0$ and $n$ is large enough, the main term in Theorem \ref{thm-full-asymptotic} will dominate the sign of $c_t^{(m)}(n)$. Now we are going to see how large $n$ should be. By Lemma \ref{lem-Bessel-ineq} we have for $n\geq \mu+\frac{1}{\mu}\left(\frac{3t}{4\pi} \right)^2$,
\begin{align}
&\left|\frac{2\pi \alpha_t^{(m)}(n)}{t}\mu^{\frac{1}{2}}(n-\mu)^{-\frac{1}{2}}I_{-1}\left(\frac{4\pi}{t}\sqrt{\mu(n-\mu)} \right)\right|\nonumber \\
&>\frac{1}{100}\sqrt{\frac{\pi}{t}}\mu^{\frac{1}{4}}(n-\mu)^{-\frac{3}{4}}\exp\left(\frac{4\pi\sqrt{\mu(n-\mu)}}{t}\right).
\end{align}
 Now we denote
\begin{align}
\Delta_t^{(m)}(n):=\frac{1}{100}\sqrt{\frac{\pi}{t}}\mu^{\frac{1}{4}}(n-\mu)^{-\frac{3}{4}}\exp\left(\frac{4\pi\sqrt{\mu(n-\mu)}}{t}\right)-\overline{E}_t^{(m)}(n)
\end{align}
where $\overline{E}_t^{(m)}(n)$ was given in \eqref{error-exp}. It is easy to show that for each positive integer $t \leq 24$, there exists a computable number $B(t)$ such that when $n>B(t)$, $\Delta_t^{(m)}(n)>0$ for any positive integer $m\leq 24/(t-1)$. For $t\in \{2,3,5,7,11,13,17,19,23\}$, with the help of Mathematica, we find an appropriate value (not the smallest one) for $B(t)$ (see Table \ref{tab-bound}).
\begin{table}[!h]
\renewcommand\arraystretch{1.3}
\begin{tabular}{|c|c|c|c|c|c|c|c|c|c|}
  \hline
  $t$ & 2 & 3 & 5 & 7 & 11 & 13 & 17 & 19 & 23  \\
  \hline
  $B(t)$ & 250  & 300 & 460 & 540 & 1910 &3430 & 7000 & 10450 &21650 \\
  \hline
\end{tabular}
\vspace{2mm}
\caption{Values of $B(t)$ for prime $t \leq 24$}\label{tab-bound}
\end{table}

Therefore, when $n>B(t)$, the sign of $c_t^{(m)}(n)$ is the same as the sign of $\alpha_t^{(m)}(n)$, which agrees with the statement in Theorem \ref{thm-sign}. For $n\leq B(t)$, we verify the sign of $c_t^{(m)}(n)$ one by one through Mathematica. Finally, we find all the exceptional numbers $n$ as the set $\mathcal{E}_t^{(m)}$ in the appendix records.

It is clear that the above procedure can be applied to other $t$ as well. That is, we can actually find the exceptional set $\mathcal{E}_t^{(m)}$ for all positive integers $t,m$ with $m(t-1)\leq 24$. But the workload of computation will be heavier (though still doable).  We leave this to the interested reader.

\section{Vanishing coefficients and three special cases}\label{sec-qseries}

Recall Ramanujan's theta functions (see \cite[Definition 1.2.1 and Corollary 1.3.4]{Berndt-text})
\begin{align}
\phi(q)&:=\sum_{n=-\infty}^\infty q^{n^2}=\frac{(q^2;q^2)_\infty^5}{(q;q)_\infty^2(q^4;q^4)_\infty^2}, \label{phi-defn} \\
\psi(q)&:=\sum_{n=0}^\infty q^{n(n+1)/2}=\frac{(q^2;q^2)_\infty^2}{(q;q)_\infty}.
\end{align}

Sometimes we will use the compressed notation
\begin{align}
(a_1,a_2,\dots,a_k;q)_\infty=(a_1;q)_\infty(a_2;q)_\infty\cdots (a_k;q)_\infty.
\end{align}

For any Laurent series
$$f(q)=\sum_{n=-\infty}^\infty a(n)q^n,$$
we use $[q^n]f(q)$ to denote the coefficient of $q^n$ in the series expansion of $f(q)$, i.e., $[q^n]f(q)=a(n)$. We may need to compare $f(q)$ with another series
$$g(q)=\sum_{n=-\infty}^\infty b(n)q^n.$$
We say that $f(q)\succeq g(q)$ or $g(q)\preceq f(q)$ if and only if $a(n)\geq b(n)$ holds for every integer $n$. For example, it is easy to see that
\begin{align}\label{eta-reciprocal}
\phi(q)\succeq 1, \quad \psi(q)\succeq 1, \quad \frac{1}{(q;q)_\infty}\succeq \frac{1}{1-q}=\sum_{n=0}^\infty q^n.
\end{align}

The Borwein brothers \cite{Borweins} introduced the following cubic theta functions:
\begin{align}
a(q)&:=\sum_{m,n=-\infty}^\infty q^{m^2+mn+n^2}, \label{a-defn}\\
b(q)&:=\sum_{m,n=-\infty}^\infty e^{2\pi i(m-n)/3}q^{m^2+mn+n^2}, \label{b-defn} \\
c(q)&:=\sum_{m,n=-\infty}^\infty q^{(m+\frac{1}{3})^2+(m+\frac{1}{3})(n+\frac{1}{3})+(n+\frac{1}{3})^2}. \label{c-defn}
\end{align}
It is known that \cite{Borweins,BBG}
\begin{align}
a(q)&=\frac{(q;q)_\infty^3}{(q^3;q^3)_\infty}+9q\frac{(q^9;q^9)_\infty^3}{(q^3;q^3)_\infty}, \label{a(q)-eta}\\
b(q)&=\frac{(q;q)_\infty^3}{(q^3;q^3)_\infty}, \label{b(q)-eta} \\
c(q)&=3q^{1/3}\frac{(q^3;q^3)_\infty^3}{(q;q)_\infty}, \label{c(q)-eta}
\end{align}
and they satisfy an elegant identity \cite{Borweins}:
\begin{align}\label{cube-id}
a^3(q)=b^3(q)+c^3(q).
\end{align}
It was also proved that \cite[Eqs.\ (2.3),(2.4),(2.11)]{Borweins}
\begin{align}
a(q)&=3a(q^3)-2\frac{(q;q)_\infty^3}{(q^3;q^3)_\infty},  \label{aq-id-1} \\
a(q)&=a(q^3)+6q\frac{(q^9;q^9)_\infty^3}{(q^3;q^3)_\infty}. \label{aq-id-2}
\end{align}
By eliminating $a(q)$ from \eqref{aq-id-1} and \eqref{aq-id-2}, we obtain
\begin{align}\label{triple-id}
\frac{(q;q)_\infty^3}{(q^3;q^3)_\infty}=a(q^3)-3q\frac{(q^9;q^9)_\infty^3}{(q^3;q^3)_\infty}.
\end{align}
\begin{lemma}\label{lem-exp}
Let $t$ be a positive integer and $0\leq r,s <t$ be integers such that
\begin{align}\label{eq-rs}
r\not\equiv \frac{j(3j+1)}{2} \pmod{t}, \quad s\not\equiv \frac{j(j+1)}{2} \pmod{t},  \quad \text{for all $0\leq j <2t$}.
\end{align}
Then
\begin{align}
c_t^{(1)}(tn+r)=0, \label{ct1-vanish}\\
c_t^{(3)}(tn+s)=0. \label{ct3-vanish}
\end{align}
\end{lemma}
\begin{proof}
The assertions \eqref{ct1-vanish} and \eqref{ct3-vanish} follow from \eqref{eq-pentagonal} and \eqref{eq-Jacobi}, respectively.
\end{proof}

Lemma \ref{lem-exp}  can be used to explain most of the assertions in \eqref{eq-zero} in the cases $m=1$ or $3$. To prove \eqref{eq-zero} completely, we still need to treat the cases $(t,m)=(4,8)$ and $(9,3)$. Lemma \ref{lem-exp} is not applicable for the first case and cannot fully explain the second case. Hence we discuss them separately.
\begin{prop}\label{prop-c48}
For any integer $n\geq 0$, we have
\begin{align}
c_4^{(8)}(2n+1)=0
\end{align}
except that $c_4^{(8)}(1)=-8$.
\end{prop}
\begin{proof}
From \cite[Eq.\ (2.20)]{WangAAM} we find
\begin{align}\label{J1-power4}
(q;q)_\infty^4=\frac{(q^4;q^4)_\infty^{10}}{(q^2;q^2)_\infty^2(q^8;q^8)_\infty^4}-4q\frac{(q^2;q^2)_\infty^2(q^8;q^8)_\infty^4}{(q^4;q^4)_\infty^2}.
\end{align}
Taking square on both sides of \eqref{J1-power4}, we obtain
\begin{align}\label{J1-power8}
(q;q)_\infty^8=\frac{(q^4;q^4)_\infty^{20}}{(q^2;q^2)_\infty^4(q^8;q^8)_\infty^8}+16q^2\frac{(q^2;q^2)_\infty^4(q^8;q^8)_\infty^8}{(q^4;q^4)_\infty^4}-8q(q^4;q^4)_\infty^8.
\end{align}
Thus
\begin{align}
\sum_{n=0}^\infty c_4^{(8)}(2n+1)q^n=-8.
\end{align}
This proves the desired assertion.
\end{proof}
\begin{prop}\label{prop-c93}
For any integer $n\geq 0$, we have
\begin{align}
c_9^{(3)}(3n+r)=0, \quad r\in\{1,2\}
\end{align}
except that $c_9^{(3)}(1)=-3$.
\end{prop}
\begin{proof}
By \eqref{triple-id} we have
\begin{align}
\frac{(q;q)_\infty^3}{(q^9;q^9)_\infty^3}=-3q+a(q^3)\frac{(q^3;q^3)_\infty}{(q^9;q^9)_\infty^3}.
\end{align}
The assertion follows immediately.
\end{proof}

The above discussion would be enough to explain \eqref{eq-zero}.  Now we discuss the three special cases in Theorem \ref{thm-zero}.
\begin{prop}\label{prop-level3}
For any integer $n\geq 0$, we have $c_3^{(9)}(3n)=c_3^{(3)}(n)$ and
\begin{align}
c_3^{(3)}(3n)>0, \quad c_3^{(3)}(3n+1)<0, \quad c_3^{(3)}(3n+2)=0.
\end{align}
\end{prop}
\begin{proof}
By \eqref{triple-id} we have
\begin{align}
\frac{(q;q)_\infty^9}{(q^3;q^3)_\infty^9}=\frac{1}{(q^3;q^3)_\infty^6}\left(a(q^3)-3q\frac{(q^9;q^9)_\infty^3}{(q^3;q^3)_\infty } \right)^3.
\end{align}
This implies
\begin{align}
\sum_{n=0}^\infty c_3^{(9)}(3n)q^n=\frac{1}{(q;q)_\infty^6}\left(a^3(q)-27q\frac{(q^3;q^3)_\infty^9}{(q;q)_\infty^3}\right).
\end{align}
By \eqref{cube-id} and \eqref{b(q)-eta}, we simplify the above and arrive at
\begin{align}
\sum_{n=0}^\infty c_3^{(9)}(3n)q^n=\frac{(q;q)_\infty^3}{(q^3;q^3)_\infty^3}.
\end{align}
Therefore, $c_3^{(9)}(3n)=c_3^{(3)}(n)$. Now using \eqref{triple-id} again we obtain $c_3^{(3)}(3n+2)=0$ and
\begin{align}
&\sum_{n=0}^\infty c_3^{(3)}(3n)q^n=\frac{a(q)}{(q;q)_\infty^2}\succeq \sum_{n=0}^\infty q^n, \\
&\sum_{n=0}^\infty c_3^{(3)}(3n+1)q^n=-3\frac{(q^3;q^3)_\infty^3}{(q;q)_\infty^3}=-3\frac{1}{(q;q^3)_\infty^3(q^2;q^3)_\infty^3}\preceq -\sum_{n=0}^\infty q^n.
\end{align}
Here for the first inequality we used the fact that  $a(q)\succeq 1$, which follows from \eqref{a-defn}.
This proves the desired assertion.
\end{proof}

\begin{prop}\label{prop-level4}
For any integer $n\geq 0$, we have
\begin{align}
c_4^{(4)}(8n)>0, \quad c_4^{(4)}(8n+2)>0, \quad c_4^{(4)}(8n+4)<0, \quad c_4^{(4)}(8n+6)<0.
\end{align}
\end{prop}
\begin{proof}
From \eqref{J1-power4} we deduce that
\begin{align}\label{c4-2n}
\sum_{n=0}^\infty c_4^{(4)}(2n)q^n=\frac{(q^2;q^2)_\infty^6}{(q;q)_\infty^2(q^4;q^4)_\infty^4}.
\end{align}
From \cite[Eq.\ (2.21)]{WangAAM} we find
\begin{align}\label{J1-minus2}
\frac{1}{(q;q)_\infty^2}=\frac{(q^8;q^8)_\infty^5}{(q^2;q^2)_\infty^5(q^{16};q^{16})_\infty^2}+2q\frac{(q^4;q^4)_\infty^2(q^{16};q^{16})_\infty^2}{(q^2;q^2)_\infty^5(q^8;q^8)_\infty}.
\end{align}
Substituting \eqref{J1-minus2} into \eqref{c4-2n}, we obtain
\begin{align}
\sum_{n=0}^\infty c_4^{(4)}(4n)q^n=\frac{(q;q)_\infty (q^4;q^4)_\infty^5}{(q^2;q^2)_\infty^4 (q^8;q^8)_\infty^2}=\frac{(q;q)_\infty}{(q^2;q^2)_\infty}\cdot \frac{1}{(q^2;q^2)_\infty} \cdot \phi(q^2), \\
\sum_{n=0}^\infty c_4^{(4)}(4n+2)q^n=2\frac{(q;q)_\infty (q^8;q^8)_\infty^2}{(q^2;q^2)_\infty^2(q^4;q^4)_\infty}=2\frac{(q;q)_\infty}{(q^2;q^2)_\infty}\cdot \frac{(q^8;q^8)_\infty^2}{(q^2;q^2)_\infty(q^4;q^4)_\infty}.
\end{align}
Therefore,
\begin{align}
&\sum_{n=0}^\infty c_4^{(4)}(8n)q^n=\frac{\phi(q)}{(q;q)_\infty}\sum_{n=0}^\infty c_2^{(1)}(2n)q^n \succeq \sum_{n=0}^\infty q^n,\\
&\sum_{n=0}^\infty c_4^{(4)}(8n+4)q^n=\frac{\phi(q)}{(q;q)_\infty}\sum_{n=0}^\infty c_2^{(1)}(2n+1)q^n \preceq -\sum_{n=0}^\infty q^n, \\
&\sum_{n=0}^\infty c_4^{(4)}(8n+2)q^n=2\frac{(q^4;q^4)_\infty^2}{(q;q)_\infty(q^2;q^2)_\infty}\sum_{n=0}^\infty c_2^{(1)}(2n)q^n \succeq \sum_{n=0}^\infty q^n,\\
&\sum_{n=0}^\infty c_4^{(4)}(8n+6)q^n=2\frac{(q^4;q^4)_\infty^2}{(q;q)_\infty(q^2;q^2)_\infty}\sum_{n=0}^\infty c_2^{(1)}(2n+1)q^n \preceq -\sum_{n=0}^\infty q^n.
\end{align}
Here we used the fact that $\sum_{n=0}^\infty c_2^{(1)}(2n)q^n \succeq 1$ and $\sum_{n=0}^\infty c_2^{(1)}(2n+1)q^n \preceq -1$, which follows from \eqref{eq-Andrews-sign} and the fact that $c_2^{(1)}(0)=1$ and $c_2^{(1)}(1)=-1$.
This completes the proof.
\end{proof}

\begin{prop}\label{prop-level5}
For any integer $n\geq 0$, we have $c_5^{(5)}(5n)=c_5^{(1)}(n)$ and
\begin{align}
c_5^{(1)}(5n)>0, \quad c_5^{(1)}(5n+1)<0, \quad c_5^{(1)}(5n+2)<0, \\
c_5^{(1)}(5n+3)=0, \quad c_5^{(1)}(5n+4)=0.
\end{align}
\end{prop}
\begin{proof}
Let
\begin{align}
R(q):=\frac{(q;q^5)_\infty(q^4;q^5)_\infty}{(q^2;q^5)_\infty(q^3;q^5)_\infty}=\sum_{n=0}^\infty x(n)q^n,
\end{align}
which is essentially the Rogers-Ramanujan continued fraction.
From \cite[Theorems 7.4.1 and 7.4.4]{Berndt-text} we have
\begin{align}
\frac{1}{R(q^5)}-q-q^2R(q^5)&=\frac{(q;q)_\infty}{(q^{25};q^{25})_\infty}, \label{RR-1} \\
\frac{1}{R^5(q^5)}-11q^5-q^{10}R^5(q^5)&=\frac{(q^5;q^5)_\infty^6}{(q^{25};q^{25})_\infty^6}. \label{RR-2}
\end{align}

Let $(q;q)_\infty^5=\sum_{n=0}^\infty y(n)q^n$. Taking fifth power on both sides of \eqref{RR-1} and collecting those terms in which the power of $q$ is divisible by $5$, we deduce from \eqref{RR-2} that
\begin{align}\label{x-gen}
\sum_{n=0}^\infty y(5n)q^n=(q^5;q^5)_\infty^5\left(\frac{1}{R^5(q)}-11q-q^2R^5(q)\right) =\frac{(q;q)_\infty^6}{(q^5;q^5)_\infty}.
\end{align}
Therefore,
\begin{align}
\sum_{n=0}^\infty c_5^{(5)}(5n)q^n=\frac{(q;q)_\infty}{(q^5;q^5)_\infty}=\sum_{n=0}^\infty c_5^{(1)}(n)q^n.
\end{align}
Hence we have $c_5^{(5)}(5n)=c_5^{(1)}(n)$. By \eqref{RR-1} we have
\begin{align}
\frac{(q;q)_\infty}{(q^5;q^5)_\infty}=\frac{(q^{25};q^{25})_\infty}{(q^5;q^5)_\infty}\left(\frac{1}{R(q^5)}-q-q^2R(q^5)\right).
\end{align}
From this dissection formula we deduce that $c_5^{(1)}(5n+r)=0$ for $r\in \{3,4\}$ and
\begin{align}
&\sum_{n=0}^\infty c_5^{(1)}(5n)q^n=\frac{(q^5;q^5)_\infty}{(q;q)_\infty}\cdot \frac{1}{R(q)}=\frac{1}{(q;q^5)_\infty^2 (q^4;q^5)_\infty^2}\succeq \sum_{n=0}^\infty q^n,\\
&\sum_{n=0}^\infty c_5^{(1)}(5n+1)q^n=-\frac{(q^5;q^5)_\infty}{(q;q)_\infty}=-\frac{1}{(q,q^2,q^3,q^4;q^5)_\infty}\preceq -\sum_{n=0}^\infty q^n,\\
&\sum_{n=0}^\infty c_5^{(1)}(5n+2)q^n=-R(q)\frac{(q^5;q^5)_\infty}{(q;q)_\infty}=-\frac{1}{(q^2;q^5)_\infty^2(q^3;q^5)_\infty^2}\preceq -\sum_{n=0}^\infty q^n.
\end{align}
The desired assertions then follow immediately.
\end{proof}

\begin{proof}[Proof of Theorem \ref{thm-zero}]
We denote the sets of $r$ and $s$ satisfying \eqref{eq-rs} by $H_t^{(1)}$ and $H_t^{(3)}$, respectively.
By direct computation, we find
\begin{align*}
&H_3^{(3)}=\{2\}, \quad H_5^{(1)}=\{3,4\}, \quad H_5^{(3)}=\{2,4\}, \quad H_6^{(3)}=\{2,5\}, \\
&H_7^{(1)}=\{3,4,6\}, \quad H_7^{(3)}=\{2,4,5\}, \quad H_9^{(3)}=\{2,4,5,7,8\}, \\
&H_{10}^{(1)}=\{3,4,8,9\}, \quad H_{11}^{(1)}=\{3,6,8,9,10\}, \quad H_{13}^{(1)}=\{3,4,6,8,10,11\}, \\
&H_{14}^{(1)}=\{3,4,6,10,11,13\}, \quad H_{15}^{(1)}=\{3,4,8,9,13,14\}, \\
&H_{17}^{(1)}=\{3,4,8,10,11,13,14,16\}, \quad H_{19}^{(1)}=\{4,6,8,9,10,11,14,17,18\}, \\
&H_{20}^{(1)}=\{3,4,8,9,13,14,18,19\}, \quad H_{21}^{(1)}=\{3,4,6,10,11,13,17,18,20\} \\
&H_{22}^{(1)}=\{3,6,8,9,10,14,17,19,20,21\}, \quad H_{23}^{(1)}=\{4,6,9,10,13,14,16,18,19,20,21\}.
\end{align*}
We find that the above sets $H_t^{(m)}$ agree with $Z_t^{(m)}$ in the appendix except that
\begin{align}
Z_9^{(3)}=H_9^{(3)}\cup \{1\}.
\end{align}
But by Proposition \ref{prop-c93} we know that $c_9^{(3)}(9n+1)=0$ except for $c_9^{(3)}(1)=-3$.

Besides the three special cases, the only nonempty set $Z_t^{(m)}$ left is $Z_4^{(8)}=\{1,3\}$, for which we need to prove that
$$c_4^{(8)}(4n+1)=0, \quad c_4^{(8)}(4n+3)=0$$
except for $c_4^{(8)}(1)=-8$. This follows from Proposition \ref{prop-c48}. Now from the above facts and Lemma \ref{lem-exp} we know that \eqref{eq-zero} holds.

Next, the three special cases \eqref{eq-exc-1}, \eqref{eq-exc-2} and \eqref{eq-exc-3} follow from Propositions \ref{prop-level3}, \ref{prop-level4} and \ref{prop-level5}, respectively.
\end{proof}

\section{Further results on the conjecture}\label{sec-conj}
In this section, we will confirm Conjecture \ref{conj} for $(t,m)\in \{(2,m),(p,1),(p,3)\}$ for arbitrary positive integer $m$ and prime $p$.

First, we generalize the result of Hu and Ye \cite{Hu-Ye} from $c_2^{(24)}(n)$ to $c_2^{(m)}(n)$ for arbitrary positive integer $m$.
\begin{theorem}\label{thm-sign-2}
For any integer $m\geq 1$ and $n\geq 0$, we have
\begin{align}
c_2^{(m)}(2n)>0, \quad c_2^{(m)}(2n+1)<0
\end{align}
except for $c_2^{(1)}(2)=0$.
\end{theorem}
We give two proofs for Theorem \ref{thm-sign-2} which do not require much knowledge. Though these proofs are quite simple, we include both of them here for the purpose that they may shed some light in solving Conjecture \ref{conj}.
\begin{proof}[First Proof of Theorem \ref{thm-sign-2}]
Replacing $q$ by $-q$ in \eqref{ctmn-defn} with $m=2$, we obtain
\begin{align}\label{c2m-gen}
\sum_{n=0}^\infty (-1)^nc_2^{(m)}(n)q^n=(-q;q^2)_\infty^m=\prod\limits_{n=0}^\infty(1+q^{2n+1})^m.
\end{align}
Thus we have for any $m\geq 1$,
\begin{align}
(-1)^nc_2^{(m)}(n)\geq (-1)^nc_2^{(1)}(n).
\end{align}
Clearly, $(-1)^nc_2^{(1)}(n)$ counts the number of partitions of $n$ into distinct odd parts. This fact was observed before by Andrews \cite[p.\ 488]{Andrews}. If $n$ is odd, then $n$ itself gives a partition of $n$. If $n\geq 4$ is even, then $1+(n-1)$ is a desired partition of $n$. Thus $(-1)^nc_2^{(1)}(n)>0$ holds for all $n\geq 1$ except that $c_2^{(1)}(2)=0$. When $m\geq 2$, by \eqref{c2m-gen} we find that $c_2^{(m)}(2)=\binom{m}{2}$. This finishes the proof.
\end{proof}
To present our second proof, we need the following simple fact.
\begin{lemma}\label{lem-alternating}
Let $f_i(q)=\sum_{n=0}^\infty a_i(n)q^n$ $(i=1,2,\cdots, m, m\geq 2)$ be series satisfying $(-1)^na_i(n)\geq 0$ (resp. $(-1)^na_i(n)> 0$) for each $n\geq 0$ (resp.\ for each $n\geq 0$ except that $a_i(2)$ is allowed to be zero). Then their product
$$f(q)=f_1(q)\cdots f_m(q)=\sum_{n=0}^\infty a(n)q^n$$
satisfies $(-1)^na(n)\geq 0$ (resp. $(-1)^na(n)> 0$) for all $n\geq 0$.
\end{lemma}
\begin{proof}
It suffices to prove the assertion for $m=2$. We write the $2$-dissections of $f_i(q)$ ($i=1,2$) as
\begin{align}
f_i(q)=A_i(q^2)-qB_i(q^2), \quad A_i(q), B_i(q)\in \mathbb{R}[[q]], \quad i=1,2.
\end{align}
Note that
\begin{align}
f(q)=&\left(A_1(q^2)-qB_1(q^2)\right)\left(A_2(q^2)-qB_2(q^2)\right) \nonumber \\
=&\left(A_1(q^2)A_2(q^2)+q^2B_1(q^2)B_2(q^2)\right)-q\left(A_1(q^2)B_2(q^2)+A_2(q^2)B_1(q^2)\right).
\end{align}

In the first case, we have $A_i(q)\succeq 0$ and $B_i(q)\succeq 0$ ($i=1,2$). It is clear that
\begin{align}
A_1(q)A_2(q)+qB_1(q)B_2(q)\succeq 0, \quad A_1(q)B_2(q)+A_2(q)B_1(q)\succeq 0.
\end{align}
Hence $(-1)^na(n)\geq 0$ holds for any $n\geq 0$.

In the second case, we have $(-1)^na_i(n)>0$ ($i=1,2$) holds for any $n\geq 0$ except that $a_i(2)$ may be zero. Hence
\begin{align*}
A_1(q),A_2(q)\succeq 1+\sum_{n=2}^\infty q^n, \quad B_1(q),B_2(q)\succeq \sum_{n=0}^\infty q^n.
\end{align*}
Hence
\begin{align*}
&A_1(q)A_2(q)+qB_1(q)B_2(q)\succeq A_1(q)+qB_1(q)\succeq \sum_{n=0}^\infty q^n, \\
&A_1(q)B_2(q)+A_2(q)B_1(q) \succeq A_1(q)+B_1(q) \succeq \sum_{n=0}^\infty q^n.
\end{align*}
Hence $(-1)^na(n)>0$ holds for any $n\geq 0$.
\end{proof}
\begin{proof}[Second proof of Theorem \ref{thm-sign-2}]
The case $m=1$ can be proved in the same way as the first proof. Thus we assume the fact that $(-1)^nc_2^{(1)}(n)>0$ except for $c_2^{(1)}(2)=0$. Now applying Lemma \ref{lem-alternating} with $f_i(q)=G_2(q)$, we have $f(q)=f_1(q)\cdots f_m(q)=G_2^m(q)$. Hence $(-1)^nc_2^{(m)}(n)>0$ for any $m\geq 2$ and $n\geq 0$.
\end{proof}

After a closer examination of Andrews' proof of \eqref{eq-Andrews-sign} in \cite{Andrews}, we can prove Conjecture \ref{conj} for the case $(t,m)=(p,1)$ for any prime $p$. Because of Theorem \ref{thm-sign-2}, we only need to consider the case $p\geq 3$.
\begin{theorem}\label{thm-cp1}
Let $p\geq 3$ be a prime. If $n\not\equiv \frac{j(3j+1)}{2}$ for any integers $j$ satisfying $\frac{1-p}{2}\leq j\leq \frac{p-1}{2}$, then $c_p^{(1)}(n)=0$. For any integers $j$ satisfying $\frac{1-p}{2}\leq j\leq \frac{p-1}{2}$ and $n\geq 2$, we have
\begin{align}
\mathrm{sgn}\left(c_p^{(1)}\Big(pn+\frac{j(3j+1)}{2}\Big)\right)=(-1)^j.
\end{align}
\end{theorem}
\begin{proof}
We will follow the lines in Andrews' proof of \eqref{eq-Andrews-sign} in \cite{Andrews} but with some modifications.

Using \eqref{eq-pentagonal} and by splitting the sum according to the residue of $n$ modulo $p$, Andrews \cite[Eq.\ (2.5)]{Andrews} proved that
\begin{align}
&\frac{(q;q)_\infty}{(q^p;q^p)_\infty}=\sum_{r=\frac{1-p}{2}}^{\frac{p-1}{2}}(-1)^rq^{r(3r+1)/2}\frac{(q^{3p^2},q^{p(3p+1)/2+3pr},q^{p(3p-1)/2-3pr};q^{3p^2})_\infty}{(q^p;q^p)_\infty} \nonumber \\
&=Q_{\frac{3p-1}{2},\frac{3p-1}{2}}(1;q^p)+\sum_{r=1}^{\frac{p-1}{2}}(-1)^rq^{r(3r-1)/2}\left(Q_{\frac{3p-1}{2},\frac{3p+1}{2}-3r}(1;q^p)  +q^rQ_{\frac{3p-1}{2},\frac{3p-1}{2}-3r}(1;q^p)\right), \label{Andrews-dissection}
\end{align}
where
\begin{align}\label{Q-AB}
Q_{k,i}(1;q):=\frac{(q^i,q^{2k+1-i},q^{2k+1};q^{2k+1})_\infty}{(q;q)_\infty}=\sum_{n\geq 0}A_{k,i}(n)q^n=\sum_{n\geq 0}B_{k,i}(n)q^n,
\end{align}
and
\begin{align}
A_{k,i}(n)&=\text{the number of partitions of $n$ into parts} \,\, \not\equiv 0,\pm i\pmod{2k+1}, \label{defn-A-partition} \\
B_{k,i}(n)&=\text{the number of partitions of $n$ of the form $n=b_1+b_2+\cdots +b_j$} \nonumber \\
 & \text{wherein $b_h\geq b_{h+1}$, $b_h-b_{h+k-1}\geq 2$, and at most $i-1$ of the $b_h$ equal 1.} \label{defn-B-partition}
\end{align}
Note that when $k\geq 2$, we have for $n\geq 2$ that
\begin{align}\label{An-positive}
A_{k,i}(n)>0.
\end{align}
In fact, when $i\not \equiv \pm 1$ (mod $2k+1$), then \eqref{An-positive} holds since $n$ has a partition consisting of 1's. If $i\equiv \pm 1$ (mod $2k+1$), then $n$ has at least a partition into $2$'s and $3$'s. Thus \eqref{An-positive} still holds.

We also recall from \cite[Eq.\ (2.12)]{Andrews-AM} that
\begin{align}\label{ineq-Andrews-B}
B_{\ell,i}(m)>B_{\ell,j}(n) \quad \text{if $\ell\geq 4$, $m>n\geq 1$ and $\ell\geq i>j>0$.}
\end{align}

Recall the set $H_p^{(1)}$ in the proof of Theorem \ref{thm-zero}. From \eqref{Andrews-dissection} we have
\begin{align}
c_p^{(1)}(pn+r)=0, \quad r \in H_p^{(1)},
\end{align}
which we have already seen in Lemma \ref{lem-exp}. Let $\overline{H_p^{(1)}}$ be the set of residues of $\frac{r(3r+1)}{2}$ modulo $p$ where $r$ ranges from $\frac{1-p}{2}$ to $\frac{p-1}{2}$. We have
\begin{align}
H_p^{(1)}\cup \overline{H_p^{(1)}}=\left\{0,1,2,\cdots,p-1\right\}.
\end{align} Now given $a\in \overline{H_p^{(1)}}$, there are at most one $r$ and one $s$ satisfying
\begin{align}
\frac{r(3r-1)}{2}\equiv a, \quad \frac{s(3s+1)}{2}\equiv a \quad \text{and} \quad 0 \leq r,s \leq \frac{p-1}{2}.
\end{align}

\textbf{Case 1.} If such $r$ exists while $s$ doest not exist, then the terms of the form $q^{pn+a}$ ($n\geq 0$) only appear in
\begin{align}
(-1)^rq^{\frac{r(3r-1)}{2}}Q_{\frac{3p-1}{2},\frac{3p+1}{2}-3r}(1;q^p).
\end{align}
From \eqref{Q-AB} and \eqref{An-positive} we deduce that
\begin{align}
\mathrm{sgn}(c_p^{(1)}(n))=(-1)^r, \quad \text{if $n\equiv a$ (mod $p$) and $n\geq 2p+\frac{r(3r-1)}{2}$.}
\end{align}

\textbf{Case 2.} If such $r$ doest not exist, then $s$ exists. Again, by \eqref{Q-AB} and \eqref{An-positive} we deduce that
\begin{align}
\mathrm{sgn}(c_p^{(1)}(n))=(-1)^s, \quad \text{if $n\equiv a$ (mod $p$) and $n\geq 2p+\frac{s(3s+1)}{2}$.}
\end{align}

\textbf{Case 3.} If both $r$ and $s$ exist, we split our discussions into two subcases.

\textbf{Case 3.1} If $a=0$, then $r=s=0$ and terms of the form $q^{pn}$ ($n\geq 0$) only appear in the term $Q_{\frac{3p-1}{2},\frac{3p-1}{2}}(1;q^p)$. By  \eqref{Q-AB} and \eqref{An-positive} we see that
\begin{align}
\mathrm{sgn}(c_p^{(1)}(n))=1, \quad \text{if $n\equiv 0$ (mod $p$) and $n\geq 2p$.}
\end{align}

\textbf{Case 3.2} If $a\neq 0$, then both $r$ and $s$ are nonzero. Terms of the form $q^{pn+a}$ only appear in the term
\begin{align}\label{eq-sum}
(-1)^rq^{r(3r-1)/2}Q_{\frac{3p-1}{2},\frac{3p+1}{2}-3r}(1;q^p)  +(-1)^sq^{s(3s+1)/2}Q_{\frac{3p-1}{2},\frac{3p-1}{2}-3s}(1;q^p).
\end{align}
From the choices of $r$ and $s$ we have
\begin{align}
\frac{r(3r-1)}{2}-\frac{s(3s+1)}{2}=\frac{(r+s)(3(r-s)-1)}{2} =\lambda p, \quad \lambda \in \mathbb{Z}.
\end{align}
Since $2\leq r+s\leq p-1$, we deduce that
\begin{align*}
3(r-s)\equiv 1 \pmod{p}.
\end{align*}
Hence $r\neq s$. If $r<s$, then we may write the sum in \eqref{eq-sum} as
\begin{align}
(-1)^rq^{r(3r-1)/2}\left(Q_{\frac{3p-1}{2},\frac{3p+1}{2}-3r}(1;q^p)  +(-1)^{r+s}q^{|\lambda| p}Q_{\frac{3p-1}{2},\frac{3p-1}{2}-3s}(1;q^p)\right).
\end{align}
If $r+s$ is even, then from the above we deduce that
\begin{align}\label{cp1-ineq-proof-1}
\mathrm{sgn}(c_p^{(1)}(n))=(-1)^r \quad \text{if  $n\equiv a \pmod{p}$ and $n\geq \frac{r(3r-1)}{2}+2p$}.
\end{align}
If $r+s$ is odd, then since $\frac{3p-1}{2}\geq 4$, from \eqref{ineq-Andrews-B} we deduce that for $n\geq 2$,
\begin{align*}
&[q^{pn}]\left(Q_{\frac{3p-1}{2},\frac{3p+1}{2}-3r}(1;q^p)  +(-1)^{r+s}q^{|\lambda| p}Q_{\frac{3p-1}{2},\frac{3p-1}{2}-3s}(1;q^p)\right) \nonumber \\
&=B_{\frac{3p-1}{2},\frac{3p+1}{2}-3r}(n)-B_{\frac{3p-1}{2},\frac{3p-1}{2}-3s}(n-|\lambda|)>0.
\end{align*}
Hence \eqref{cp1-ineq-proof-1} still holds.

If $r>s$, then we may write the sum in \eqref{eq-sum} as
\begin{align}
(-1)^sq^{s(3s+1)/2}\left(Q_{\frac{3p-1}{2},\frac{3p-1}{2}-3s}(1;q^p)+(-1)^{r+s}q^{|\lambda| p} Q_{\frac{3p-1}{2},\frac{3p+1}{2}-3r}(1;q^p) \right).
\end{align}
Arguing as before, we deduce that
\begin{align}\label{cp1-ineq-proof-2}
\mathrm{sgn}(c_p^{(1)}(n))=(-1)^s \quad \text{if  $n\equiv a \pmod{p}$ and $n\geq \frac{s(3s+1)}{2}+2p$}.
\end{align}
Combining the above cases, we complete our proof.
\end{proof}

From the work of Schlosser and Zhou \cite{Schlosser-Zhou}, we can also prove the conjecture for $(t,m)=(p,3)$ with $p$ being any prime. Again, we only need to consider the case $p\geq 3$.
\begin{theorem}\label{thm-cp3}
Let $p\geq 3$ be a prime. If $n\not\equiv \frac{j(j+1)}{2}$ $\pmod{p}$ for any integer $0\leq j \leq \frac{p-1}{2}$, then $c_p^{(3)}(n)=0$. Furthermore, for any integers $j$ satisfying $0\leq j \leq \frac{p-1}{2}$ and $n\geq 0$, we have
\begin{align}
\mathrm{sgn}\left(c_p^{(3)}\Big(pn+\frac{j(j+1)}{2}\Big)\right)=(-1)^j.
\end{align}
\end{theorem}
\begin{proof}
The assertions follow readily from the proof of \eqref{eq-Schlosser-sign} in \cite{Schlosser-Zhou}. We omit the details.
\end{proof}

\begin{proof}[Proof of Theorem \ref{thm-conj-partial}]
This follows from Theorems \ref{thm-sign-2}, \ref{thm-cp1} and \ref{thm-cp3}.
\end{proof}


\begin{rem}
After the completion of the first version of this article, we learned from Prof.\ Dongxi Ye that under his guidance, Yansu Wang, a student at the Sun Yat-Sen University (Zhuhai) has independently discovered the sign-change behaviors of $c_p^{24/(p-1)}(n)$ with $p\in \{3,5,7,13\}$ in her 2021 undergraduate thesis.
\end{rem}

\subsection*{Acknowledgements}
We thank Prof. Wentson J.T. Zang for some helpful comments. The author was supported by the National Natural Science Foundation of China (11801424) and a start-up research grant of the Wuhan University.

\section*{Appendix}
In this appendix, we provide the sets $P_t^{(m)}$, $N_t^{(m)}$ and $Z_t^{(m)}$ for $1\leq t \leq 24$ and $1\leq m \leq \frac{24}{t-1}$ in explicit forms. As said in Theorem \ref{thm-zero}, there are three special cases corresponding to $(t,m)=(3,9), (4,4)$ and $(5,5)$. We will add marks (SC1), (SC2) and (SC3) for these cases in the tables and point out what this special case means at the bottom of that table. Moreover, for $t$ being a prime less than $24$, we also give the exceptional set $\mathcal{E}_t^{(m)}$ explicitly.

\begin{table}[H]
\renewcommand\arraystretch{1.3}
\begin{tabular}{cccccc}
  \hline
 $m$ & $P_{2}^{(m)}$ & $N_{2}^{(m)}$ & $Z_{2}^{(m)}$  & $\mathcal{E}_{2}^{(m)}$ \\
 1  & $\{0\}$  & $\{1\}$ &  $\emptyset$ & $\{2\}$ \\
 $m\geq 2$ & $\{0\}$ & $\{1\}$ & $\emptyset$ & $\emptyset$ \\
 \hline
\end{tabular}
\caption{Data for $c_{2}^{(m)}(n)$}\label{tab-2}
\end{table}

\begin{table}[H]
\renewcommand\arraystretch{1.3}
\begin{tabular}{cccccc}
  \hline
 $m$ & $P_{3}^{(m)}$ & $N_{3}^{(m)}$ & $Z_{3}^{(m)}$  & $\mathcal{E}_{3}^{(m)}$ \\
  1 & $\{0\}$ & $\{1,2\}$ & $\emptyset$  & $\{5\}$ \\
  2 & $\{0\}$ & $\{1,2\}$ & $\emptyset$ & $\{5\}$ \\
  3 & $\{0\}$ & $\{1\}$ & $\{2\}$ & $\emptyset$ \\
  4 & $\{0\}$ & $\{1,2\}$ & $\emptyset$ & $\emptyset$ \\
  5 & $\{0,2\}$ & $\{1\}$ & $\emptyset$ & $\emptyset$ \\
  6 & $\{0,2\}$ & $\{1\}$ & $\emptyset$ & $\{0\}$ \\
  7 & $\{0,2\}$ & $\{1\}$ & $\emptyset$ & $\{0\}$ \\
  8 & $\{0,2\}$ & $\{1\}$ & $\emptyset$ & $\{0\}$ \\
  9 & $\{2\}$ & $\{1\}$ & $\{0\}$ (SC1) & $\emptyset$  \\
  10 & $\{2\}$ & $\{0,1\}$ & $\emptyset$ & $\{0\}$ \\
  11 & $\{2\}$ & $\{0,1\}$ & $\emptyset$ & $\{0\}$ \\
  12 & $\{2\}$ & $\{0,1\}$ & $\emptyset$ & $\{0\}$ \\
  \hline
\end{tabular}

SC1: $c_3^{(9)}(9n)>0$, $c_3^{(9)}(9n+3)<0$ and $c_3^{(9)}(9n+6)=0$ (see \eqref{eq-exc-1}).
\caption{Data for $c_{3}^{(m)}(n)$}\label{tab-3}
\end{table}

\begin{table}[H]
\renewcommand\arraystretch{1.3}
\begin{tabular}{ccccc}
  \hline
 $m$ & $P_{4}^{(m)}$ & $N_{4}^{(m)}$ & $Z_{4}^{(m)}$   \\
  1 & $\{0,3\}$ & $\{1,2\}$ & $\emptyset$   \\
  2 & $\{0,3\}$ & $\{1,2\}$ & $\emptyset$ \\
  3 & $\{0,3\}$ & $\{1,2\}$ & $\emptyset$  \\
  4 & $\{3\}$ & $\{1\}$ & $\{0,2\}$ (SC2) \\
  5 & $\{2,3\}$ & $\{0,1\}$ & $\emptyset$  \\
  6 & $\{2,3\}$ & $\{0,1\}$ & $\emptyset$  \\
  7 & $\{2,3\}$ & $\{0,1\}$ & $\emptyset$  \\
  8 & $\{2\}$ & $\{0\}$ & $\{1,3\}$  \\
  \hline
\end{tabular} \\
SC2: $c_4^{(4)}(8n+r)>0$ ($r\in \{0,2\}$) and $c_4^{(4)}(8n+r)<0$ ($r\in \{4,6\}$) (see \eqref{eq-exc-2}).
\caption{Data for $c_{4}^{(m)}(n)$}\label{tab-4}
\end{table}

\begin{table}[H]
\renewcommand\arraystretch{1.3}
\begin{tabular}{cccccc}
  \hline
 $m$ & $P_{5}^{(m)}$ & $N_{5}^{(m)}$ & $Z_{5}^{(m)}$  & $\mathcal{E}_{5}^{(m)}$ \\
  1 & $\{0\}$ & $\{1,2\}$ & $\{3,4\}$  & $\{7\}$ \\
  2 & $\{0,3,4\}$ & $\{1,2\}$ & $\emptyset$ & $\{9\}$ \\
  3 & $\{0,3\}$ & $\{1\}$ & $\{2,4\}$ & $\emptyset$ \\
  4 & $\{0,2,3\}$ & $\{1,4\}$ & $\emptyset$ & $\{5\}$ \\
  5 & $\{2,3\}$ & $\{1,4\}$ & $\{0\}$ (SC3) & $\{30\}$ \\
  6 & $\{0,2,3\}$ & $\{1,4\}$ & $\emptyset$ & $\emptyset$ \\
  \hline
\end{tabular}

SC3: $c_5^{(5)}(25n)>0$, $c_5^{(5)}(25n+r)<0$ ($r\in \{5,10\}$) and $c_5^{(5)}(25n+r)<0$ ($r\in \{15,20\}$) (see \eqref{eq-exc-3}).
\caption{Data for $c_{5}^{(m)}(n)$}\label{tab-5}
\end{table}

\begin{table}[H]
\renewcommand\arraystretch{1.3}
\begin{tabular}{ccccc}
  \hline
 $m$ & $P_{6}^{(m)}$ & $N_{6}^{(m)}$ & $Z_{6}^{(m)}$ \\
  1 & $\{0,4,5\}$ & $\{1,2,3\}$ & $\emptyset$   \\
  2 & $\{3,4,5\}$ & $\{0,1,2\}$ & $\emptyset$  \\
  3 & $\{3,4\}$ & $\{0,1\}$ & $\{2,5\}$  \\
  4 & $\{2,3,4\}$ & $\{0,1,5\}$ & $\emptyset$  \\
  \hline
\end{tabular}
\caption{Data for $c_{6}^{(m)}(n)$}\label{tab-6}
\end{table}

\begin{table}[H]
\renewcommand\arraystretch{1.3}
\begin{tabular}{cccccc}
  \hline
 $m$ & $P_{7}^{(m)}$ & $N_{7}^{(m)}$ & $Z_{7}^{(m)}$  & $\mathcal{E}_{7}^{(m)}$ \\
  1 & $\{0,5\}$ & $\{1,2\}$ & $\{3,4,6\}$  & $\{12\}$ \\
  2 & $\{0,3,4,5\}$ & $\{1,2,6\}$ & $\emptyset$ & $\emptyset$ \\
  3 & $\{0,3\}$ & $\{1,6\}$ & $\{2,4,5\}$ & $\emptyset$ \\
  4 & $\{0,2,3\}$ & $\{1,4,5,6\}$ & $\emptyset$ & $\emptyset$ \\
  \hline
\end{tabular}
\caption{Data for $c_{7}^{(m)}(n)$}\label{tab-7}
\end{table}

\begin{table}[H]
\renewcommand\arraystretch{1.3}
\begin{tabular}{ccccc}
  \hline
 $m$ & $P_{8}^{(m)}$ & $N_{8}^{(m)}$ & $Z_{8}^{(m)}$   \\
  1 & $\{0,5,6,7\}$ & $\{1,2,3,4\}$ & $\emptyset$  \\
  2 & $\{0,3,4,5\}$ & $\{1,2,6,7\}$ & $\emptyset$  \\
  3 & $\{0,2,3,5\}$ & $\{1,4,6,7\}$ & $\emptyset$ \\
  \hline
\end{tabular}
\caption{Data for $c_{8}^{(m)}(n)$}
\end{table}

\begin{table}[H]
\renewcommand\arraystretch{1.3}
\begin{tabular}{ccccc}
  \hline
 $m$ & $P_{9}^{(m)}$ & $N_{9}^{(m)}$ & $Z_{9}^{(m)}$  \\
  1 & $\{0,4,5,7,8\}$ & $\{1,2,3,6\}$ & $\emptyset$  \\
  2 & $\{0,3,4,5,7\}$ & $\{1,2,6,8\}$ & $\emptyset$ \\
  3 & $\{0,3\}$ & $\{6\}$ & $\{1,2,4,5,7,8\}$  \\
  \hline
\end{tabular}
\caption{Data for $c_{9}^{(m)}(n)$}
\end{table}

\begin{table}[H]
\renewcommand\arraystretch{1.3}
\begin{tabular}{ccccc}
  \hline
 $m$ & $P_{10}^{(m)}$ & $N_{10}^{(m)}$ & $Z_{10}^{(m)}$   \\
  1 & $\{0,6,7\}$ & $\{1,2,5\}$ & $\{3,4,8,9\}$   \\
  2 & $\{0,3,4,6,7\}$ & $\{1,2,5,8,9\}$ & $\emptyset$\\
  \hline
\end{tabular}
\caption{Data for $c_{10}^{(m)}(n)$}
\end{table}

\begin{table}[H]
\renewcommand\arraystretch{1.3}
\begin{tabular}{cccccc}
  \hline
 $m$ & $P_{11}^{(m)}$ & $N_{11}^{(m)}$ & $Z_{11}^{(m)}$  & $\mathcal{E}_{11}^{(m)}$ \\
  1 & $\{0,5,7\}$ & $\{1,2,4\}$ & $\{3,6,8,9,10\}$  & $\emptyset$ \\
  2 & $\{0,3,4,5,7,10\}$ & $\{1,2,6,8,9\}$ & $\emptyset$  & $\emptyset$ \\
  \hline
\end{tabular}
\caption{Data for $c_{11}^{(m)}(n)$}\label{tab-11}
\end{table}

\begin{table}[H]
\renewcommand\arraystretch{1.3}
\begin{tabular}{ccccc}
  \hline
 $m$ & $P_{12}^{(m)}$ & $N_{12}^{(m)}$ & $Z_{12}^{(m)}$   \\
  1 & $\{0,5,7,8,9,10\}$ & $\{1,2,3,4,6,11\}$ & $\emptyset$   \\
  2 & $\{0,2,3,4,5,7\}$ & $\{1,6,8,9,10,11\}$ & $\emptyset$  \\
  \hline
\end{tabular}
\caption{Data for $c_{12}^{(m)}(n)$}
\end{table}

\begin{table}[H]
\renewcommand\arraystretch{1.3}
\begin{tabular}{cccccc}
  \hline
 $m$ & $P_{13}^{(m)}$ & $N_{13}^{(m)}$ & $Z_{13}^{(m)}$  & $\mathcal{E}_{13}^{(m)}$ \\
  1 & $\{0,5,7,9\}$ & $\{1,2,12\}$ & $\{3,4,6,8,10,11\}$  & $\{9,35\}$ \\
  2 & $\{0,3,4,5,10\}$ & $\{1,2,6,7,8,9,11,12\}$ & $\emptyset$  & $\{7,11,12,23,25\}$ \\
  \hline
\end{tabular}
\caption{Data for $c_{13}^{(m)}(n)$}\label{tab-13}
\end{table}

\begin{table}[H]
\renewcommand\arraystretch{1.3}
\begin{tabular}{ccccc}
  \hline
 $m$ & $P_{14}^{(m)}$ & $N_{14}^{(m)}$ & $Z_{14}^{(m)}$   \\
  1 & $\{0,5,8,9\}$ & $\{1,2,7,12\}$ & $\{3,4,6,10,11,13\}$  \\
  \hline
\end{tabular}
\caption{Data for $c_{14}^{(m)}(n)$}
\end{table}

\begin{table}[H]
\renewcommand\arraystretch{1.3}
\begin{tabular}{ccccc}
  \hline
 $m$ & $P_{15}^{(m)}$ & $N_{15}^{(m)}$ & $Z_{15}^{(m)}$  \\
  1 & $\{0,5,6,7,11\}$ & $\{1,2,10,12\}$ & $\{3,4,8,9,13,14\}$  \\
  \hline
\end{tabular}
\caption{Data for $c_{15}^{(m)}(n)$}
\end{table}

\begin{table}[H]
\renewcommand\arraystretch{1.3}
\begin{tabular}{ccccc}
  \hline
 $m$ & $P_{16}^{(m)}$ & $N_{16}^{(m)}$ & $Z_{16}^{(m)}$  \\
  1 & $\{0,4,5,6,7,9,10,11\}$ & $\{1,2,3,8,12,13,14,15\}$ & $\emptyset$  \\
  \hline
\end{tabular}
\caption{Data for $c_{16}^{(m)}(n)$}
\end{table}

\begin{table}[H]
\renewcommand\arraystretch{1.3}
\begin{tabular}{cccccc}
  \hline
 $m$ & $P_{17}^{(m)}$ & $N_{17}^{(m)}$ & $Z_{17}^{(m)}$  & $\mathcal{E}_{17}^{(m)}$ \\
  1 & $\{0,5,7,9\}$ & $\{1,2,6,12,15\}$ & $\{3,4,8,10,11,13,14,16\}$  & $\{6,9,23,57\}$ \\
  \hline
\end{tabular}
\caption{Data for $c_{17}^{(m)}(n)$}\label{tab-17}
\end{table}

\begin{table}[H]
\renewcommand\arraystretch{1.3}
\begin{tabular}{ccccc}
  \hline
 $m$ & $P_{18}^{(m)}$ & $N_{18}^{(m)}$ & $Z_{18}^{(m)}$   \\
  1 & $\{0,3,4,5,6,7,8,10,11\}$ & $\{1,2,9,12,13,14,15,16,17\}$ & $\emptyset$  \\
  \hline
\end{tabular}
\caption{Data for $c_{18}^{(m)}(n)$}\label{tab-18}
\end{table}

\begin{table}[H]
\renewcommand\arraystretch{1.3}
\begin{tabular}{cccccc}
  \hline
 $m$ & $P_{19}^{(m)}$ & $N_{19}^{(m)}$ & $Z_{19}^{(m)}$  & $\mathcal{E}_{19}^{(m)}$ \\
  1 & $\{0,3,5,7,13\}$ & $\{1,2,12,15,16\}$ & $\begin{matrix} \{4,6,8,9,10,\\ 11,14,17,18\} \end{matrix}$  & $\{3,13,16,32,70\}$ \\
  \hline
\end{tabular}
\caption{Data for $c_{19}^{(m)}(n)$}\label{tab-19}
\end{table}

\begin{table}[H]
\renewcommand\arraystretch{1.3}
\begin{tabular}{ccccc}
  \hline
 $m$ & $P_{20}^{(m)}$ & $N_{20}^{(m)}$ & $Z_{20}^{(m)}$  \\
  1 & $\{0,2,5,6,7,11\}$ & $\{1,10,12,15,16,17\}$ & $\{3,4,8,9,13,14,18,19\}$   \\
  \hline
\end{tabular}
\caption{Data for $c_{20}^{(m)}(n)$}\label{tab-20}
\end{table}

\begin{table}[H]
\renewcommand\arraystretch{1.3}
\begin{tabular}{ccccc}
  \hline
 $m$ & $P_{21}^{(m)}$ & $N_{21}^{(m)}$ & $Z_{21}^{(m)}$  \\
  1 & $\{0,5,7,8,9,16\}$ & $\{1,2,12,14,15,19\}$ & $\{3,4,6,10,11,13,17,18,20\}$   \\
  \hline
\end{tabular}
\caption{Data for $c_{21}^{(m)}(n)$}\label{tab-21}
\end{table}

\begin{table}[H]
\renewcommand\arraystretch{1.3}
\begin{tabular}{ccccc}
  \hline
 $m$ & $P_{22}^{(m)}$ & $N_{22}^{(m)}$ & $Z_{22}^{(m)}$   \\
  1 & $\{0,4,5,7,13\}$ & $\{1,2,11,12,15,16,18\}$ & $\{3,6,8,9,10,14,17,19,20,21\}$   \\
  \hline
\end{tabular}
\caption{Data for $c_{22}^{(m)}(n)$}\label{tab-22}
\end{table}

\begin{table}[H]
\renewcommand\arraystretch{1.3}
\begin{tabular}{cccccc}
  \hline
 $m$ & $P_{23}^{(m)}$ & $N_{23}^{(m)}$ & $Z_{23}^{(m)}$ & $\mathcal{E}_{23}^{(m)}$ \\
  1 & $\{0,3,5,7,11,22\}$ & $\{1,2,8,12,15,17\}$ & $\begin{matrix} \{4,6,9,10,13,14,\\ 16,18,19,20,21\}\end{matrix}$  & $\begin{matrix}\{3,8,11,17, \\ 31, 34,54,100\}\end{matrix} $ \\
  \hline
\end{tabular}
\caption{Data for $c_{23}^{(m)}(n)$}\label{tab-23}
\end{table}

\begin{table}[H]
\renewcommand\arraystretch{1.3}
\begin{tabular}{ccccc}
  \hline
 $m$ & $P_{24}^{(m)}$ & $N_{24}^{(m)}$ & $Z_{24}^{(m)}$   \\
  1 & $\begin{matrix}\{0,3,4,5,7,9,13,\\14,18,20,22,23\}\end{matrix}$ & $\begin{matrix}\{1,2,6,8,10,11,12,\\ 15,16,17,19,21\}\end{matrix}$ & $\emptyset$  \\
  \hline
\end{tabular}
\caption{Data for $c_{24}^{(m)}(n)$}\label{tab-24}
\end{table}


\begin{thebibliography}{0}
\bibitem{AS} M. Abramowitz, I.A. Stegun (Eds.), Handbook of Mathematical Functions with Formulas, Graphs,
and Mathematical Tables, 10th printing, United States Department of Commerce, National Bureau of Standards, 1972.

\bibitem{Andrews-AM} G.E. Andrews, Ramanujan's ``lost'' notebook III. The Rogers-Ramanujan continued fraction, Adv. Math. 41 (1981) 186--208.

\bibitem{Andrews}  G.E. Andrews, On a conjecture of Peter Borwein. J. Symbolic Comput. 20 (1995) 487--501.



\bibitem{Notebook3} B.C. Berndt, Ramanujan's Notebooks, Part III, Springer-Verlag, New York, 1991.

\bibitem{Berndt-text} B.C. Berndt, Number Theory in the Spirit of Ramanujan, AMS, Providence (2006).

\bibitem{Borweins} J.M. Borwein and P.B. Borwein, A cubic counterpart of Jacobi's identity and the AGM, Trans. Amer. Math. Soc. 323 (1991) 691--701.

\bibitem{BBG} J.M. Borwein, P.B. Borwein and F.G. Garvan, Some cubic modular identities of Ramanujan, Trans. Amer. Math. Soc. 343(1) (1994) 35--47.

\bibitem{Chern} S. Chern, Asymptotics for the Fourier coefficients of eta-quotients, J. Number Theory 199 (2019) 168--191.

\bibitem{CTW} S. Chern, D. Tang and L. Wang, Some inequalities for Garvan's bicrank function of 2-colored partitions, Acta Arith. 190 (2019) 171--191.


\bibitem{Hu-Ye} B. Hu and D. Ye, Sign-change of the Fourier coefficients of a hauptmodul for $\Gamma_0(2)$, Int. J. Number Theory 14(8) (2018) 2269-2276.

\bibitem{Hu-Ye-preparation} B. Hu and D. Ye, Sign changes of Fourier coefficients of Hauptmoduls for genus zero groups, in preparation.

\bibitem{MO}  T. Matsusaka and R. Osanai, Arithmetic formulas for the Fourier coefficients of Hauptmoduln of level 2, 3, and 5, Proc. Amer. Math. Soc. 145 (2017) 1383--1392.

\bibitem{Ohta} K. Ohta, Formulas for the Fourier coefficients of some genus zero modular functions, Kyushu J. Math. 63 (2009) 1--15.

\bibitem{Schlosser} M.J. Schlosser, A tribute to Dick Askey, arXiv:1909.10508, Sep. 2019.


\bibitem{Schlosser-Zhou} M.J. Schlosser and N.H. Zhou, On the infinite Borwein product raised to a positive real power, arXiv:2011.10552v1.



\bibitem{WangAAM} L. Wang, Arithmetic properties of odd ranks and $k$-marked odd Durfee symbols, Adv. Appl. Math 121 (2020) 102098.
\end{thebibliography}
\end{document}